\documentclass[11pt,reqno,twoside]{article}
\usepackage{amsmath,amsfonts,amsthm,amssymb}
\usepackage{graphics}
\usepackage[all]{xy}
\usepackage{color}
\usepackage{multicol}
\usepackage{amscd}
\theoremstyle{definition}
\newtheorem{theorem}{Theorem}[section]
\newtheorem{proposition}[theorem]{Proposition}

\newtheorem{lemma}[theorem]{Lemma}

\newtheorem{definition}[theorem]{Definition}
\newtheorem{example}[theorem]{Example}

\renewcommand{\qed}{\hfill$\square$}
\newcommand{\lt}[2]{#1~\underline{\triangleright}~#2}
\newcommand{\ut}[2]{#1~\overline{\triangleright}~#2}
\begin{document}

\title{\large\bf 
BIQUANDLE COHOMOLOGY AND STATE-SUM INVARIANTS OF LINKS AND SURFACE-LINKS
}

\author{ 
{Seiichi Kamada, Akio Kawauchi, Jieon Kim and Sang Youl Lee}
\bigskip\\
{\small\it 
Department of Mathematics, Osaka City University,
}\\ 
{\small\it  Sugimoto, Sumiyoshi-ku, Osaka 558-8585, Japan}\\
{\small\it skamada@sci.osaka-cu.ac.jp}
\smallskip\\
{\small\it 
Osaka City University Advanced Mathematical Institute, Osaka City University,
}\\ 
{\small\it  Sugimoto, Sumiyoshi-ku, Osaka 558-8585, Japan}\\
{\small\it kawauchi@sci.osaka-cu.ac.jp}
\smallskip\\
{\small\it 
Department of Mathematics, Pusan National University,
}\\ 
{\small\it Busan 46241, Republic of Korea}\\
{\small\it jieonkim@pusan.ac.kr}
\smallskip\\
{\small\it 
Department of Mathematics, Pusan National University,
}\\ 
{\small\it Busan 46241, Republic of Korea}\\
{\small\it sangyoul@pusan.ac.kr}}

\renewcommand\leftmark{\centerline{\footnotesize 
S. Kamada, A. Kawauchi, J. Kim \& S. Y. Lee}}
\renewcommand\rightmark{\centerline{\footnotesize 
Biquandle cohomology and state-sum invariants of links and surface-links
}}

\date{}

\maketitle

\begin{abstract}
In this paper, we discuss the (co)homology theory of biquandles, derived biquandle cocycle invariants for oriented surface-links using broken surface diagrams and how to compute the biquandle cocycle invariants from marked graph diagrams. We also develop the shadow (co)homology theory of biquandles and construct the shadow biquandle cocycle invariants for oriented surface-links.
\end{abstract}

%%%%%

% main text
\section{Introduction}\label{sect-intr}
%\label{intro}

In \cite{Jo}, D. Joyce introduced an algebraic structure known as a {\it quandle}, which is a set $X$ with a binary operation satisfying certain conditions coming from oriented Reidemeister moves for oriented link diagrams (see also \cite{Ma}). Quandles were generalized to racks in \cite{FeRo} and racks were generalized to biracks in \cite{FeRoSa1} and a (co)homology theory for racks and biracks was introduced. In the quandle case, a subcomplex was defined corresponding to Reidemeister moves of type I and this leads to the quandle (co)homology theory and quandle cocycle invariants of links and surface-links in \cite{CJKLS}. In \cite{CKS}, the {\it shadow quandle cocycle invariants} was also defined for links and surface-links. These invariants are defined as the state-sums over all quandle colorings of arcs and sheets and corresponding Boltzman weights that are the evaluations of a given quandle 2 and 3-cocycle at the crossings and triple points in a link diagram and broken surface diagram, respectively. In \cite{KKL}, S. Kamada, J. Kim and S. Y. Lee developed an interpretation of the quandle and shadow quandle cocycle invariants of surface-links in terms of marked graph diagram presentation of surface-links. 

On the other hand, a generalization of quandles (called biquandles) is introduced in \cite{KaRa}. A {\it  biquandle} is an algebraic structure with two binary operations satisfying certain conditions which can be presented by semi-arcs of (virtual) links (or semi-sheets of surface-links) as its generators modulo oriented Reidemeister moves (or Roseman moves). In \cite{CES}, J. S. Carter, M. Elhamdadi and M. Saito introduced a (co)homology theory for the set-theoretic Yang-Baxter equations and cocycles are used to define invariants via colorings of (virtual) link diagrams by biquandles and a state-sum formulation. In \cite{NeRo}, S. Nelson and J. Rosenfield introduced a generalization of biquandle homology to the case of an {\it involutory} biquandle (also known as a {\it bikei}), called bikei homology, and used bikei $2$-cocycles to enhance the bikei counting invariant for unoriented knots and links as well as unoriented and non-orientable surface-links. 

In this paper, we discuss the (co)homology theory of biquandles and develop the biquandle cocycle invariants for oriented surface-links by using broken surface diagrams generalizing quandle cocycle invariants. Then we show how to compute the biquandle cocycle invariants from marked graph diagrams. Further, we develop the shadow biquandle (co)homology theory and construct the shadow biquandle cocycle invariants for oriented surface-links presented by broken surface diagrams and also marked graph diagrams generalizing shadow quandle cocycle invariants. We also discuss a method of computing the shadow biquandle cocycle invariants from marked graph diagrams.

This paper is organized as follows. In Section \ref{sect-mgd}, we review two presentations of surface-links, broken surface diagrams and marked graph diagrams.
  In Section \ref{sect-coh}, we recall the definition of a biquandle and examples. In Section \ref{sect-bqcs}, we review the fundamental biquandles and discuss biquandle colorings for marked graph diagrams and broken surface diagrams. In Section \ref{sect-bqht}, we review the (co)homology groups of biquandles. In Section \ref{sect-cisl}, the biquandle cocycle invariants of oriented links and surface-links presented by broken surface diagrams are introduced. In Section \ref{sect-qcocm}, we introduce a method of computing biquandle $3$-cocycle invariants from marked graph diagrams. In Section \ref{sect-scicl} we develope the shadow (co)homology theory of biquandles and construct the shadow biquandle cocycle invariants of oriented surface-links.  In Section \ref{sect-sqcocm}, we introduce a method of computing shadow biquandle $3$-cocycle invariants from marked graph diagrams.

%%%%%

% main text
\section{Presentations of surface-links}
\label{sect-mgd}

A {\it surface-link} is a closed surface smoothly (or piecewise linearly and locally flatly) embedded in the Euclidian $4$-dimensional space $\mathbb R^4$. Two surface-links $\mathcal L$ and $\mathcal L'$ are {\it equivalent} if they are ambient isotopic. That is, equivalently, there exists an orientation preserving diffeomorphism (or PL homeomorphism)  $h:\mathbb R^4\rightarrow \mathbb R^4$ such that $h(\mathcal L)=\mathcal L'.$ When $\mathcal L$ and $\mathcal L'$ are oriented, it is assumed that the restriction $h|_{\mathcal L}:\mathcal L\rightarrow \mathcal L'$ is also an orientation preserving homeomorphism. Throughout this paper a surface-link means an oriented surface-link unless otherwise stated.

Let $f:F\rightarrow \mathbb R^4$ be a smooth embedding of a closed surface $F$ in $\mathbb R^4$ and let $p:\mathbb R^4\rightarrow \mathbb R^3$ be the orthogonal projection onto an affine subspace, identified with $\mathbb R^3$,  which does not intersect $\mathcal L=f(F).$ By deforming the map $f$ slightly by an ambient isotopy of $\mathbb R^4$ if necessary, we may assume that the map $p\circ f : F \to \mathbb R^3$ is a generic map. This means that for any $y=p(f(x)) \in \mathbb R^3$ there is a neighborhood $N(y) \subset \mathbb R^3$ and a diffeomorphism $\psi: N(y) \rightarrow \mathbb R^3$ such that the image of $p(f(F))\cap N(y)$ under $\psi$ looks like the intersection of $1,2,$ $3$ coordinate planes or the cone on a  figure eight (Whitney umbrella) as depicted in Fig.~\ref{fig-gen}. In these cases, the point $y$ is called a {\it non-singular point}, a {\it double point}, a {\it triple point}, or a {\it branch point}, respectively. A {\it surface-link diagram}, or simply a {\it diagram}, of a surface-link $\mathcal L=f(F)$ is the image $p(\mathcal L)$ equipped with ``over/under" information on the multiple points with respect to the direction of the projection $p$. 

\begin{figure}[ht]
\begin{center}
\resizebox{0.8\textwidth}{!}{%
  \includegraphics{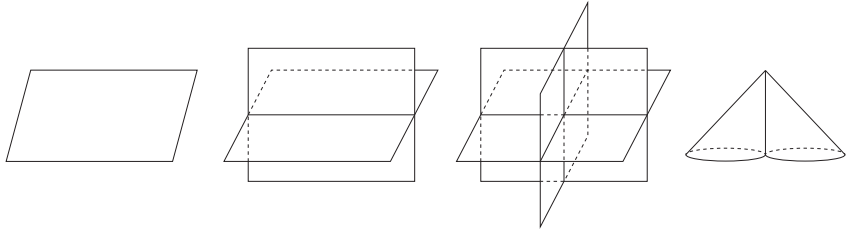}}
\caption{The images of $p(f(F))\cap N(y)$ under $\psi$}\label{fig-gen}
\end{center}
\end{figure}

A common way to indicate over/under information on a diagram is as follows. Along the double point curves on $p(\mathcal L)$, one of the sheets (called the {\it over-sheet}) lies farther than the other (called the {\it under-sheet}) with respect to the projection direction. The under-sheets are coherently broken in the projection image, and such a broken surface is called a {\it broken surface diagram} of $\mathcal L$. We call the connected components of over-sheets and under-sheets coherently broken along the double point curves in the projection {\it semi-sheets}. Consequently any surface-link gives rise to a broken surface diagram. On the other hand, for a given broken surface diagram, an embedding of a closed surface in $\mathbb R^4$ can be constructed. 
When the surface-link is oriented, we take normal vectors $\vec{n}$ to the projection of the surface such that the triple $(\vec{v}_1,\vec{v}_2,\vec{n})$ matches the orientation of $\mathbb R^3$, where $(\vec{v}_1,\vec{v}_2)$ defines the orientation of the surface. Such normal vectors are defined on the projection at all points other than the isolated branch points. For more details, see \cite{CS1}.

\begin{figure}[ht]
\begin{center}
\resizebox{0.8\textwidth}{!}{%
  \includegraphics{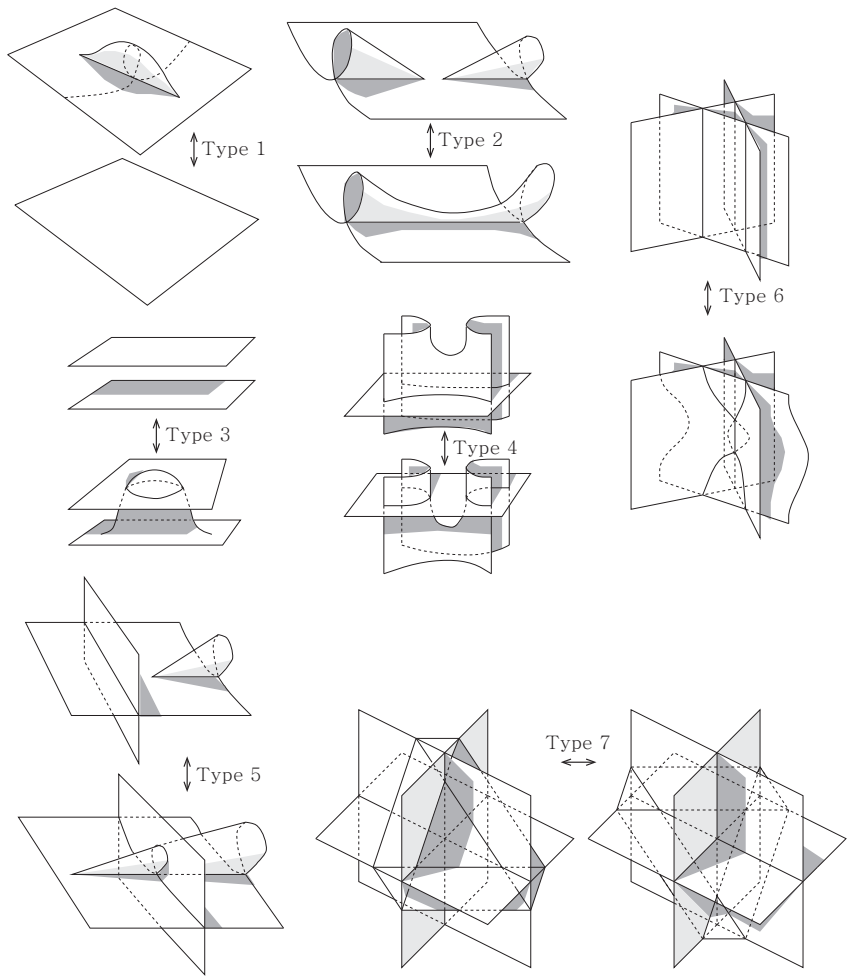}}
\caption{Roseman moves on broken surface diagrams}\label{fig-rose}
\end{center}
\end{figure}

 In \cite{Ro}, D. Roseman introduced local moves of seven types on broken surface diagrams called {\it Roseman moves} as depicted in Fig.~\ref{fig-rose}, where over/under information is omitted for simplicity. 
 
\begin{theorem}[\cite{Ro}]\label{thm-RosM}
Let $\mathcal L$ and $\mathcal L'$ be two surface-links and let $D$ and $D'$ be broken surface diagrams of $\mathcal L$ and $\mathcal L'$, respectively. Then $\mathcal L$ and $\mathcal L'$ are equivalent if and only if $D$ and $D'$ are transformed into each other by a finite sequence of ambient isotopies of $\mathbb R^3$ and Roseman moves.
 \end{theorem}

Now we review marked graph presentation of surface-links. A {\it marked graph} is a spatial graph $G$ in $\mathbb R^3$ satisfying that $G$ is a finite graph with $4$-valent vertices, say $v_1, v_2,\ldots, v_n$; each $v_i$ is a rigid vertex, that is, we fix a rectangular neighborhood $N_i$ homeomorphic to $\{(x, y)|-1 \leq x, y \leq 1\},$ where the vertex $v_i$ corresponds to the origin and the edges incident to $v_i$ are represented by $x^2 = y^2$; each $v_i$ has a {\it marker}, which is the line segment on $N_i$ represented by $\{(x, 0)|-1 \leq x \leq 1\}$. In this paper, a classical link in $\mathbb R^3$ is regarded as a marked graph without marked vertices. An {\it orientation} of a marked graph $G$ is a choice of an orientation for each edge of $G$ such that every vertex in $G$ looks like 
\xy (-5,5);(5,-5) **@{-}, 
(5,5);(-5,-5) **@{-}, 
(3,3.2)*{\llcorner}, 
(-3,-3.4)*{\urcorner}, 
(-2.5,2)*{\ulcorner},
(2.5,-2.4)*{\lrcorner}, 
(3,-0.2);(-3,-0.2) **@{-},
(3,0);(-3,0) **@{-}, 
(3,0.2);(-3,0.2) **@{-}, 
\endxy or 
\xy (-5,5);(5,-5) **@{-}, 
(5,5);(-5,-5) **@{-},  
(2.5,2)*{\urcorner}, 
(-2.5,-2.2)*{\llcorner}, 
(-3.2,3)*{\lrcorner},
(3,-3.4)*{\ulcorner},
(3,-0.2);(-3,-0.2) **@{-},
(3,0);(-3,0) **@{-}, 
(3,0.2);(-3,0.2) **@{-}, 
\endxy.  
A marked graph $G$ is said to be 
{\it orientable} if it admits an orientation. Otherwise, it is said to be {\it non-orientable}. As usual, a marked graph $G$ in $\mathbb R^3$ is described by a diagram (called a {\it marked graph diagram}) $D$ in $\mathbb R^2$ which is a generic projection on $\mathbb R^2$ with over/under crossing information for each double point  such that the restriction to a rectangular neighborhood of each marked vertex is an embedding. Fig.~\ref{fig-nori-mg} shows an oriented marked graph diagram and a non-orientable marked graph diagram. Throughout this paper, a marked graph (diagram) means an oriented marked graph (diagram), unless otherwise stated.

\begin{figure}[ht]
\begin{center}
\resizebox{0.5\textwidth}{!}{%
  \includegraphics{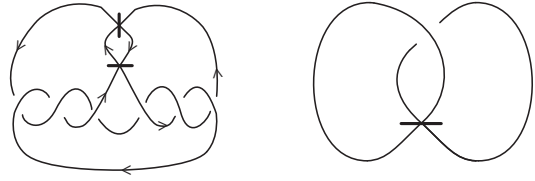}}
\caption{Marked graph diagrams}\label{fig-nori-mg}
\end{center}
\end{figure}

A surface-link $\mathcal L$ in $\mathbb R^4$ can be described in terms of its {\it cross-sections} $\mathcal L_t=\mathcal L \cap (\mathbb R^3\times \{t\})$,  $ t \in \mathbb R$ (motion pictures).  Let $p:\mathbb R^4 \to \mathbb R$ be the projection given by $p(x_1, x_2, x_3, x_4)=x_4$, and we denote by 
 $p_{\mathcal L} : \mathcal L \to  \mathbb R$ the restriction of $p$ to ${\mathcal L}$.  
 It  is known (\cite{KSS},\cite{Lo}) that any surface-link $\mathcal L$ is equivalent to a surface-link $\mathcal L'$, called a {\it hyperbolic splitting} of $\mathcal L$,
such that
the projection $p_{\mathcal L'}: \mathcal L' \to \mathbb R$ satisfies that 
all critical points are non-degenerate, 
all the index 0 critical points (minimal points) are in $\mathbb R^3\times \{-1\}$, 
all the index 1 critical points (saddle points) are in $\mathbb R^3\times \{0\}$, and 
all the index 2 critical points (maximal points) are in $\mathbb R^3\times \{1\}$.

Let $\mathcal L$ be a surface-link and let ${\mathcal L'}$ be a hyperbolic splitting of $\mathcal L.$ The cross-section $\mathcal L'_0=\mathcal L'\cap (\mathbb R^3 \times \{0\})$ at $t=0$ is a $4$-valent graph $G$ in $\mathbb R^3\times \{0\}$. We give a marker at each $4$-valent vertex (saddle point) that indicates how the saddle point opens up above as illustrated in Fig.~\ref{sleesan2:fig1}. We assume that the cross-section
$\mathcal L'_0$ has the induced orientation as the boundary of the oriented surface $\mathcal L' \cap (\mathbb R^3 \times (-\infty, 0])$. The resulting oriented marked graph $G$ is called a {\it marked graph} presenting $\mathcal L$. A diagram of $G$ is called a {\it marked graph diagram} presenting $\mathcal L$.  

\begin{figure}[ht]
\begin{center}
\resizebox{0.60\textwidth}{!}{%
  \includegraphics{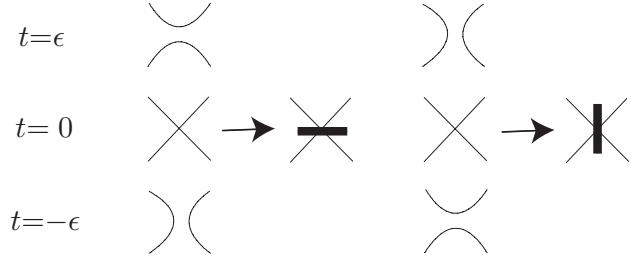} }
\caption{Marking of a vertex} \label{sleesan2:fig1}
\end{center}
\end{figure}

Let $D$ be a marked graph diagram and $D_0$ the singular link diagram obtained from $D$ by removing all markers. Let $V(D)=\{v_1,v_2,\ldots,v_n\}$ be the set of all vertices of $D$.  For each $i$ $(i=1,\ldots, n),$ consider four points $v_i^1,v_i^2,v_i^3$, and $v_i^4$ on $D$ in a neighborhood of $v_i$ as in Fig.~\ref{fig-ver}. We define
 \begin{align*}
D_+=\bigl[D_0\setminus \overset{n}{\underset{i=1}{\cup}}\bigl(\overset{4}{\underset{j=1}{\cup}}|v_i,v_i^j|\bigr)\bigr]\cup\bigl[\overset{n}{\underset{i=1}{\cup}}\bigl(|v_i^1,v_i^2|\cup|v_i^3,v_i^4|\bigr)\bigr],
\end{align*}
which is called the {\it positive resolution} of $D,$ and
 \begin{align*}
D_-=\bigl[D_0\setminus \overset{n}{\underset{i=1}{\cup}}\bigl(\overset{4}{\underset{j=1}{\cup}}|v_i,v_i^j|\bigr)\bigr]\cup\bigl[\overset{n}{\underset{i=1}{\cup}}\bigl(|v_i^1,v_i^3|\cup|v_i^2,v_i^4|\bigr)\bigr],
\end{align*}
the {\it negative resolution} of $D$, where $|v,w|$ is the line segment connecting $v$ and $w$.  
When both resolutions $D_-$ and $D_+$ are diagrams of trivial links, we say that 
$D$ is  {\it admissible}. 

\begin{figure}[ht]
\begin{center}
\resizebox{0.55\textwidth}{!}{%
  \includegraphics{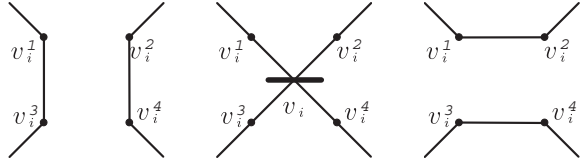} }
\caption{$v_i^1,v_i^2,v_i^3$, and $v_i^4$}\label{fig-ver}
\end{center}
\end{figure}

When $D$ is admissible, we construct a surface-link as follows (cf.  \cite{KSS},\cite{Yo}).
Let $L_0$ be a spatial graph in $\mathbb R^3$ whose diagram is $D_0$.  Let $w_i^j$ and $w_i$ be the points on $L_0$ such that $\pi(w_i^j)=v_i^j$ and $\pi(w_i)=v_i,$ respectively, where $\pi:\mathbb R^3\rightarrow \mathbb R^2$ is the projection $(x,y,z) \mapsto (x,y)$. 
For each $t \in [0,1]$, let $w_i^j(t)$ be the point $ (1-t) w_i + t w_i^j \in \mathbb{R}^3$.  

For each $t \in [0,1]$, let 
$L_t^{+}$ and $L_t^{-}$ be links in $\mathbb R^3$ defined by
$$L_t^{+} = \bigl[ L_0 \setminus \overset{n}{\underset{i=1}{\cup}} \bigl( \overset{4}{\underset{j=1}{\cup}} 
| w_i, w_i^j(t)  |   \bigr) \bigr]   \cup \bigl[ \overset{n}{\underset{i=1}{\cup}} \bigl( 
| w_i^1(t), w_i^2(t) | \cup
| w_i^3(t), w_i^4(t) | \bigr)\bigr],$$ 
$$L_{t}^{-} = \bigl[ L_0 \setminus \overset{n}{\underset{i=1}{\cup}} \bigl(\overset{4}{\underset{j=1}{\cup}}
 | w_i, w_i^j(t) | \bigr)\bigr] 
\cup \bigl[ \overset{n}{\underset{i=1}{\cup}} \bigl( 
| w_i^1(t), w_i^3(t) | \cup 
| w_i^2(t), w_i^4(t) | \bigr)\bigr]. $$ 
 
Put $L_+ = L_1^{+}$ and $L_- = L_{1}^{-}$.  Then $L_+$ and $L_-$ have diagrams $D_+$ and $D_-$, respectively.  
Let $B_1^{+},\ldots,B_\mu^+$ be mutually disjoint 2-disks in $\mathbb{R}^3$ with 
$\partial(B_1^{+}\cup\cdots\cup B_\mu^+)=L_+$, and let 
$B_1^{-},\ldots,B_\lambda^-$ be mutually disjoint 2-disks  in $\mathbb{R}^3$ with 
$\partial(B_1^{-}\cup\cdots\cup B_\lambda^-)=L_-$.  
Let $F(D)$ be a surface-link in $\mathbb{R}^4$ defined by
\begin{align*}
F(D) = \/ 
&(B_1^{-}\cup\cdots\cup B_\lambda^-)\times\{-2\} 
\cup L_- \times(-2,-1)\\
&\cup (\cup_{t\in[-1,0)}L_{-t}^-\times\{t\})
\cup L_0 \times\{0\} 
\cup (\cup_{t\in(0,1]}L_t^+\times\{t\})\\
&\cup L_+ \times(1,2) 
\cup (B_1^{+}\cup\cdots\cup B_\mu^+)\times \{2\}. 
\end{align*}

We say that $F(D)$ is a {\it surface-link associated with $D$}.  It is uniquely determined from $D$ up to equivalence (see \cite{KSS}).
A surface-link $\mathcal L$ is said to be {\it presented} by a marked graph diagram $D$ if $\mathcal L$ is equivalent to the  surface-link $F(D)$. Any surface-link can be presented by an admissible marked graph diagram. Moreover, we have 

\begin{theorem}[\cite{KK},\cite{KJL2},\cite{Sw}]\label{thm-YosM}
Let $\mathcal L$ and $\mathcal L'$ be two surface-links and let $D$ and $D'$ be marked graph diagrams presenting $\mathcal L$ and $\mathcal L'$, respectively. Then $\mathcal L$ and $\mathcal L'$ are equivalent if and only if $D$ and $D'$ are transformed into each other by a finite sequence of ambient isotopies of $\mathbb R^2$ and Yoshikawa moves.
 \end{theorem}

In \cite{As}, S. Ashihara  introduced a method of constructing a broken surface diagram of a surface-link from its marked graph diagram. For our use below, we review his construction. In what follows, by $D\rightarrow D'$ we mean that a link diagram $D'$ is obtained from a  link diagram $D$ by a single Reidemeister move $R_i(1\leq i\leq 3)$ shown in Fig.~\ref{fig-rmove} or an ambient isotopy of $\mathbb R^2$.

\begin{figure}[ht]
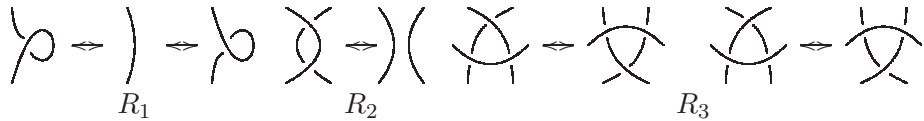

\centerline{
\xy (0,0);(4,7) **\crv{(2,7)}, (4,3);(4,7)
**\crv{(6,3)&(6,7)}, (0,10);(1.5,6) **\crv{(0.2,6.5)},
(2.5,4);(4,3) **\crv{(3,3.1)},
\endxy
 % \quad
\xy (2,5);(6,5) **@{-} ?>*\dir{>} ?<*\dir{<},(0,0)*{},
\endxy
%\quad
\hskip 0.2cm
\xy (0,0);(0,10) **\crv{(2,5)},
(1,-3) *{R_1},
\endxy
%\quad
\xy (2,5);(6,5) **@{-} ?>*\dir{>} ?<*\dir{<},(0,0)*{},
\endxy
%\quad
\hskip 0.2cm
\xy (0,10);(4,3) **\crv{(2,3)}, (4,7);(4,3)
**\crv{(6,7)&(6,3)}, (0,0);(1.5,4) **\crv{(0.2,3.5)},
(2.5,6);(4,7) **\crv{(3,6.9)},
\endxy
%%%%%
\quad
\xy (0,0);(0,10) **\crv{(9,5)}, (6,10);(3.8,8.5) **\crv{(5,9.5)},
(6,0);(3.8,1.5) **\crv{(5,0.5)}, (2.3,2.8);(2.3,7.2)
**\crv{(0.5,5)},
\endxy
 %\quad
\xy (2,5);(6,5) **@{-} ?>*\dir{>} ?<*\dir{<},
(4,-3) *{R_2},(0,0)*{},
\endxy
\xy (0,0);(0,10) **\crv{(4,5)}, 
(6,0);(6,10) **\crv{(2,5)},
\endxy
%%%%%
\quad
\xy (0,5);(10,5) **\crv{(5,0)}, (7.5,4);(2,10) **\crv{(6.7,8)},
(2,0);(2.2,2.2) **\crv{(2,1)}, (8,0);(7.8,2.2) **\crv{(8,1)},
(2.5,4);(4.4,7.7) **\crv{(3,6.5)}, (5.8,8.8);(8,10) **\crv{(6,9)},
\endxy
  %\quad
\xy (2,5);(6,5) **@{-} ?>*\dir{>} ?<*\dir{<},
,(0,0)*{},
\endxy
  %\quad
\hskip 0.2cm
\xy (0,5);(10,5) **\crv{(5,10)}, (2.5,6);(8,0) **\crv{(3.3,2)},
(2,10);(2.2,7.8) **\crv{(2,9)}, (8,10);(7.8,7.8) **\crv{(8,9)},
(7.5,6);(5.6,2.3) **\crv{(7,3.5)}, (4.2,1.2);(2,0) **\crv{(4,1)},
\endxy
%%%%%%
\xy (4,-3) *{R_{3}},(0,0)*{},
\endxy
\xy (0,5);(10,5) **\crv{(5,0)}, (2.5,4);(8,10) **\crv{(3.3,8)},
(8,0);(7.8,2.2) **\crv{(8,1)}, (2,0);(2.2,2.2) **\crv{(2,1)},
(7.5,4);(5.6,7.7) **\crv{(7,6.5)}, (4.2,8.8);(2,10) **\crv{(4,9)},
\endxy
  %\quad
\xy (2,5);(6,5) **@{-} ?>*\dir{>} ?<*\dir{<},
,(0,0)*{},
\endxy
 % \quad
\hskip 0.2cm
\xy (0,5);(10,5) **\crv{(5,10)}, (7.5,6);(2,0) **\crv{(6.7,2)},
(8,10);(7.8,7.8) **\crv{(8,9)}, (2,10);(2.2,7.8) **\crv{(2,9)},
(2.5,6);(4.4,2.3) **\crv{(3,3.5)}, (5.8,1.2);(8,0) **\crv{(6,1)},
\endxy
}\caption{Reidemeister moves}
\label{fig-rmove}
\end{figure}
  
Let $D$ be an admissible marked graph diagram and let $D_+$ and $D_-$ be the positive and the negative resolutions of $D,$ respectively.  
Since $D_+$ is a diagram of a trivial link, there is a sequence of link diagrams from $D_+$ to a trivial link diagram $O$ related by ambient isotopies of $\mathbb R^2$ and Reidemeister moves: 
$$D_+=D_1\rightarrow D_2\rightarrow\cdots\rightarrow D_r=O.$$ 
For each $i$ ($i=1, \ldots, r-1$), let $\{f^{(i)}_t\}_{t\in I}$ be a $1$-parameter  family of homeomorphisms from $\mathbb{R}^3$ to $\mathbb{R}^3$ that satisfies 
$$f_0^{(i)}={\rm id}, \quad f_1^{(i)}(L(D_i))=L(D_{i+1}),$$ 
where $L(D_i)$ denotes a link in $\mathbb{R}^3$ whose diagram is $D_i$ ($i=1, \ldots, r$). 
Without loss of generality, we may assume that $L(D_1)= L(D_+)= L_+$ and the  
following two conditions are satisfied.  
\begin{itemize}
\item[$\bullet$] When the move $D_i\rightarrow D_{i+1}$ is an ambient isotopy of $\mathbb R^2$, let $\{h^{(i)}_t\}_{t\in I}$ be  an ambient isotopy of $\mathbb R^2$ such that $h_1^{(i)}(D_i)=D_{i+1}.$ Then $f^{(i)}_t$ satisfies $\pi(f_t^{(i)}(L(D_i)))=h_t^{(i)}(\pi(L(D_i)))=h_t^{(i)}(D_i)$ for $t\in I.$ 
\item[$\bullet$] When the move $D_i\rightarrow D_{i+1}$ is a Reiedemeister move, let $B_{(i)}$ be a disk in $\mathbb R^2$ where the move is applied and let $M_{(i)}$ be the subset of $B_{(i)}\times I$ $\subset \mathbb R^3$ determined by $\pi(M_{(i)}\cap (B_{(i)}\times \{t\}))=\pi(f_t^{(i)}(L(D_i)))\cap B_{(i)}$ for $t\in I.$ 
Then $M_{(i)}$ is as shown in Fig.~\ref{fig-m1}, \ref{fig-m2}, or \ref{fig-m3}.
\end{itemize} 
 \begin{figure}[ht]
\begin{center}
\resizebox{0.5\textwidth}{!}{%
  \includegraphics{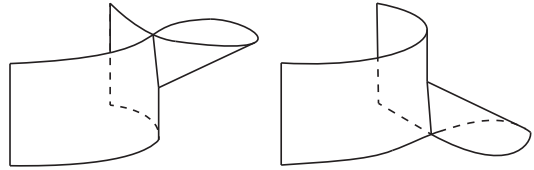} }
\caption{$M_{(i)}$ for $R_1$}
\label{fig-m1}
\end{center}
\end{figure}
 \begin{figure}[ht]
\begin{center}
\resizebox{0.5\textwidth}{!}{%
  \includegraphics{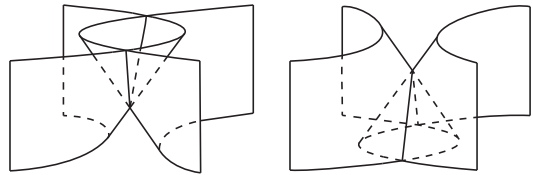} }
\caption{$M_{(i)}$ for $R_2$}
\label{fig-m2}
\end{center}
\end{figure}
 \begin{figure}[ht]
\begin{center}
\resizebox{0.5\textwidth}{!}{%
  \includegraphics{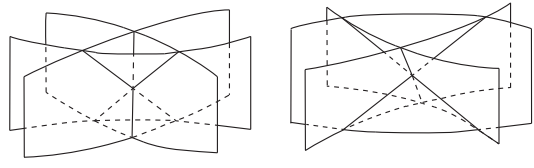} }
\caption{$M_{(i)}$ for $R_3$}
\label{fig-m3}
\end{center}
\end{figure}
Take real numbers $t_1, \ldots, t_r$ with $1 < t_1< \cdots<t_r<2.$ For each $i$ ($i=1, \ldots, r-1$), we define a homeomorphism $F^{(i)}:\mathbb{R}^4 (=\mathbb{R}^3 \times \mathbb{R}) 
\rightarrow\mathbb{R}^4$ by 
\begin{align*}
F^{(i)}(x,t)=\begin{cases}
(x,t) \hskip 2cm \text{for }t\leq t_i;\\
(f_{\phi(t)}^{(i)}(x),t) \hskip 1cm \text{for }t_i<t< t_{i+1};\\
(f_{1}^{(i)}(x),t) \hskip 1.15cm \text{for }t\geq t_{i+1}
\end{cases}
\end{align*}
for all $x\in \mathbb R^3$ and $t\in \mathbb R,$ where $\phi(t)=(t-t_i)/(t_{i+1}-t_i).$

Similarly, consider a sequence of link diagrams from $D_-$ to a trivial link diagram $O'$ related by ambient isotopies of $\mathbb R^2$ and Reidemeister moves:
$$D_-=D_1' \rightarrow D_2' \rightarrow\cdots\rightarrow D_s'=O'.$$
For each $j$ ($j=1, \ldots,s-1$), let $\{g^{(j)}_t\}_{t\in I}$ be a $1$-parameter family of homeomorphisms from $\mathbb{R}^3$ which satisfies $$ g_0^{(j)}={\rm id}, \quad g_1^{(j)}(L(D_j'))=L(D_{j+1}').$$  
Without loss of generality, we may assume that $L(D_1')=L(D_-)=L_-$ and 
the following two conditions are satisfied.  
\begin{itemize}
\item[$\bullet$] When the move $D_j'\rightarrow D_{j+1}'$ is an ambient isotopy of $\mathbb R^2$, let $\{{h'_t}^{(j)}\}_{t\in I}$ be  an ambient isotopy of $\mathbb R^2$ such that ${h'_1}^{(j)}(D_j')=D_{j+1}'.$ Then $g^{(j)}_t$ satisfies $\pi(g_t^{(j)}(L(D_j')))={h'_t}^{(j)}(\pi(L(D_j')))$ for $t\in I.$
\item[$\bullet$] When the move $D_j'\rightarrow D_{j+1}'$ is a  Reidemeister move, let $B_{(j)}'$ be a disk in $\mathbb R^2$ where the move is applied and let $M_{(j)}'$ be the subset of $B_{(j)}'\times I$ $\subset \mathbb R^3$ determined by $\pi(M_{(j)}'\cap (B_{(j)}'\times \{t\}))=\pi(g_t^{(j)}(L(D_j')))\cap B_{(j)}'$ for $t\in I.$ Then $M_{(j)}'$ is as shown in Fig.~\ref{fig-m1}, \ref{fig-m2}, or \ref{fig-m3}.
\end{itemize} 
Take real numbers $t_1', \ldots, t_s'$ with $-1 > t_1'> \cdots> t_s' > -2.$ For each $j$ ($j=1, \ldots, s-1$), we define a homeomorphism $G^{(j)}:\mathbb{R}^4(=\mathbb{R}^3 \times \mathbb{R})\rightarrow\mathbb{R}^4$ by 
\begin{align*}
G^{(j)}(x,t)=\begin{cases}
(x,t) \hskip 2cm (t\geq t_j'),\\
(g_{\psi(t)}^{(j)}(x),t) \hskip 1cm (t_{j+1}'<t< t_j'),\\
(g_{1}^{(j)}(x),t) \hskip 1.15cm (t\leq t_{j+1}'),
\end{cases}
\end{align*}
where $\psi(t)=(t_j'-t)/(t_j'-t_{j+1}').$
Let $$F'=(G^{(s-1)}\circ G^{(s-2)}\circ\cdots\circ G^{(1)}\circ F^{(r-1)}\circ F^{(r-2)}\circ\cdots\circ F^{(1)})(F(D)).$$
Then $F'$ is equivalent to $F(D)$.  

Let $B_1,\ldots,B_\mu$ be mutually disjoint 2-disks in $\mathbb{R}^3$ such that $\partial(B_1\cup\cdots\cup B_\mu)=L(O)$ and $\pi|_{B_1\cup\cdots\cup B_\mu}$ is an embedding. Let $B_1',\ldots,B_\lambda'$ be mutually disjoint 2-disks in $\mathbb{R}^3$ 
such that $\partial(B_1'\cup\cdots\cup B_\lambda')=L(O')$ and $\pi|_{B_1'\cup\cdots\cup B_\lambda'}$ is an embedding. Finally, we define $F$ to be the surface constructed as follows:
$$F=(B_1'\cup\cdots\cup B_\lambda')\times\{-2\} \cup  (F'\cap(\mathbb{R}^3\times(-2,2)) ) \cup (B_1\cup\cdots\cup B_\mu)\times\{2\}. $$
It is in general position with respect to the projection $q : \mathbb R^4 \to \mathbb R^3$ defined by $ (x,y,z,t) \mapsto (x,y,t)$. The broken surface diagram of $F$ obtained from  $q(F)$   is called a {\it broken surface diagram associated with the marked graph diagram $D$} and denoted by  $\mathcal{B}(D)$.

%%%

\section{Quandles and biquandles}
\label{sect-coh}

In this section we review the definitions and examples of quandles and biquandles. 

\begin{definition}[\cite{Jo},\cite{Ma}]
A {\it quandle} is a nonempty set $X$ with a binary operation $\ast:X\times X\rightarrow X$ satisfying that 
\begin{itemize}  
\item[(Q1)] For any $x\in X$, $x\ast x=x$.
\item[(Q2)] 
There exists a binary operation ${\ast}^{-1} :X\times X\rightarrow X$ such that for any $x, y\in X$,
$({x} \ast {y})~{\ast}^{-1}y=x$ and $({x} \ast^{-1} {y}) \ast y=x$. 
\item[(Q3)] For any $x,y,z\in X$, $(x\ast y)\ast z=(x\ast z)\ast (y\ast z)$.
\end{itemize}
A {\it rack} is a set with a binary operation that satisfies (Q2) and (Q3). 
\end{definition}

\begin{definition}[\cite{FeRoSa1},\cite{KaRa}]\label{defn-biqdle}
A {\it biquandle} is a nonempty set $X$ with two binary operations $\underline{\triangleright} : X\times X\rightarrow X~\text{and}~\overline{\triangleright}:X\times X\rightarrow X$ satisfying the following axioms:
\begin{itemize}  
\item[(B1)] For any $x\in X,$ $\lt{x}{x}=\ut{x}{x}$.
\item[(B2)] There exist two binary operations $\underline{\triangleright}^{-1}, \overline{\triangleright}^{-1}:X\times X\rightarrow X$ such that for any $x, y\in X,$ 
$(\lt{x}{y})~\underline{\triangleright}^{-1}y=x$, 
$(\ut{x}{y})~\overline{\triangleright}^{-1}y=x$, 
$\lt{(x~\underline{\triangleright}^{-1}y)}{y}=x$, and 
$\ut{(x~\overline{\triangleright}^{-1}y)}{y}=x$.
\item[(B3)] The map $H:X\times X\rightarrow X\times X$ defined by $(x,y)\mapsto (\ut{y} {x},\lt{x} {y})$ is invertible. 
\item[(B4)] For any $x,y,z\in X,$ 
\begin{align*}
&\lt{(\lt{x}{y})}{(\lt{z}{y})}=\lt{(\lt{x}{z})}{(\ut{y}{z})},\\
&\ut{(\lt{x}{y})}{(\lt{z}{y})}=\lt{(\ut{x}{z})}{(\ut{y}{z})},\\
&\ut{(\ut{x}{y})}{(\ut{z}{y})}=\ut{(\ut{x}{z})}{(\lt{y}{z})}.
\end{align*}
\end{itemize}
\end{definition}

A {\it birack} is a nonempty set $X$ with two binary operations $\underline{\triangleright}, \overline{\triangleright} : X\times X\rightarrow X$ that satisfies the axioms (B2), (B3) and (B4) above.

\begin{figure}[ht]
\begin{center}
\resizebox{0.85\textwidth}{!}{%
  \includegraphics{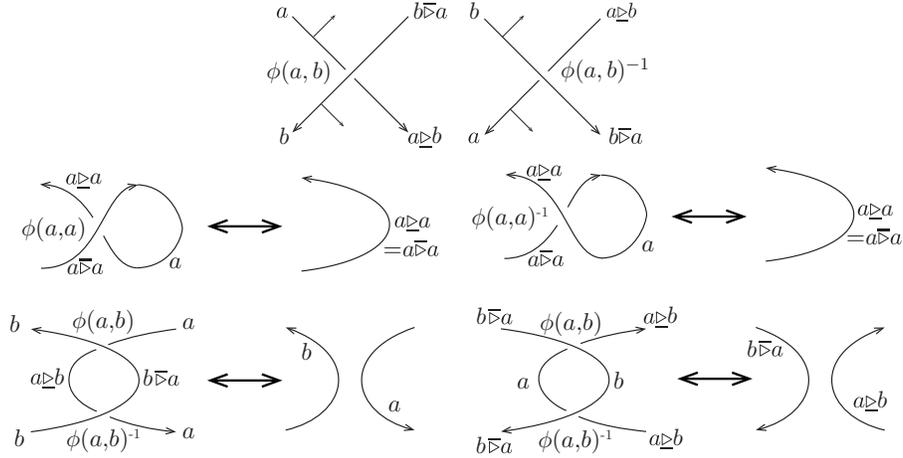}}
\caption{Reidemeister moves of Type I and II and biquandle axioms}\label{fig-clcol}
\end{center}
\end{figure}

\begin{figure}[ht]
\begin{center}
\resizebox{0.7\textwidth}{!}{%
  \includegraphics{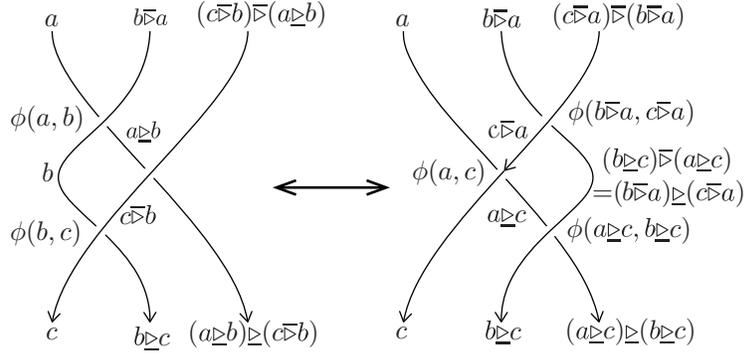}}
\caption{Reidemeister move of Type III and biquandle axioms}\label{fig-r12}
\end{center}
\end{figure}

We remind that the biquandle axioms come from the oriented Reidemeister moves. 
We divide a knot or link diagram $D$ at every crossing point (considered as a $4$-valence vertex) to obtain a collection of {\it semi-arcs}. We think of elements of a biquandle $X$ as labels for the semi-arcs in $D$ with different operations at positive and negative crossings as illustrated in the top of Fig.~\ref{fig-clcol}. The axioms are then transcriptions of a minimal set of oriented Reidemeister moves which are sufficient to generate any other oriented Reidemeister move (cf. \cite{Po}). The axiom (B1) comes from the Reidemeister move of type I as illustrated in Fig.~\ref{fig-clcol}. The axioms (B2) and (B3) come from the direct and reverse Reidemeister moves of type II respectively as illustrated in Fig.~\ref{fig-clcol}. The axiom (B4) comes from the oriented Reidemeister move of type III with all positive crossings as illustrated in Fig.~\ref{fig-r12}.

\begin{example}\label{exp-1}
Let $(X,\ast)$ be a quandle. Define $\lt{x}{y}=x\ast y$ and $\ut{x}{y}=x$ for all $x, y \in X$. Then $X$ is a biquandle. Any group $G$ is a biquandle with $\lt{x}{y}=y^{-n}xy$ and $\ut{x}{y}=x$ as well as with $\lt{x}{y}=x$ and $\ut{x}{y}=y^{-n}xy$, where $n \in \mathbb Z$.
\end{example}

A binary operation $*:X\times X\to X$ is {\it trivial} if $a*b=a$ for each $a,b\in X.$ A biquandle with trivial $\underline{\triangleright}$ (or $\overline{\triangleright}$) is really a quandle, which we call a {\it quandle biquandle}. Otherwise, it is called a {\it non-quandle biquandle}. 

\begin{example}\label{exp-2}
 Let $X$ be a module over a commutative ring $R$. Let $s$ and $t$ be invertible elements of $R$. Define
$$\lt{x}{y}=t(x-y)+s^{-1}y~\text{and}~\ut{x}{y}=s^{-1}x.$$ Then $X$ is a biquandle. In particular, any module over the two-variable Laurent polynomial ring $R=\mathbb Z[t^{\pm 1}, s^{\pm 1}]$ is a biquandle, which is called an {\it Alexander biquandle} (see \cite{KJL},\cite{KaRa}).
\end{example}

Let $X=\{x_1,x_2,\ldots,x_n\}$ be a finite biquandle. The {\it matrix} of $X$, denoted by $M_X$, is defined to be the block matrix:
$$M_X=\left[\begin{array}{c|c}
M^1&M^2
\end{array} \right],$$ where $M^1=(m^1_{ij})_{1\leq i,j\leq n}$ and $M^2=(m^2_{ij})_{1\leq i,j\leq n}$ are $n\times n$ matrices with entries in $X$ given by $$m^k_{ij}=\left\{
\begin{array}{ll}
 \lt {x_i}{x_j} & \hbox{for $k=1$,} \\
\ut {x_i}{x_j}& \hbox{for $k=2$.} 
\end{array}%
\right.$$ 

\begin{example}\label{exp-3}
The Alexander biquandle $X=\mathbb Z_4=\{0,1,2,3\}$ in Example \ref{exp-2} modulo $4$ with $s=3$ and $t=1$ is a non-quandle biquandle with the biquandle matrix:
\begin{equation*}
M_X=\left[\begin{array}{cccc|cccc} 
3& 1& 3& 1& 3& 3& 3& 3\\
0& 2& 0& 2& 2& 2& 2& 2\\
1& 3& 1& 3& 1& 1& 1& 1\\
2& 0& 2& 0& 0& 0& 0& 0
\end{array} \right]. 
\end{equation*}
\end{example}

Let $X$ and $Y$ be biquandles. A function $f:X\rightarrow Y$ is called a {\it biquandle homomorphism} if $f(\lt{x}{y})=\lt{f(x)}{f(y)}$ and $f(\ut{x}{y})=\ut{f(x)}{f(y)}$ for any $x,y\in X.$ We denote the set of all biquandle homomorphisms from $X$ to $Y$ by ${\rm Hom}(X,Y)$. A bijective biquandle homomorphism is called a {\it biquandle isomorphism}. Two biquandles $X$ and $Y$ are said to be {\it isomorphic} if there is a biquandle isomorphism $f: X \to Y$. In \cite{NeVo}, S. Nelson and J. Vo classified biquandles of order 2, 3, and 4 up to biquandle isomorphism. There are 2 biquandles of order 2, 10 biquandles of order 3, and 57 non-quandle biquandles of order 4.

%%%

\section{Biquandle colorings of diagrams}\label{sect-bqcs}

Let $D$ be a marked graph diagram and let $C(D)$ and $V(D)$ denote the set of all crossings and marked vertices of $D$, respectively. By a {\it semi-arc} of $D$ we mean  a connected component of $D \setminus (C(D) \cup V(D))$.  (At a crossing of $D$ the under-arcs are assumed to be cut.) Let $S(D)$ denote the set of semi-arcs of $D$. Since $D$ is oriented, we assume that it is co-oriented: The {\it co-orientation} of a semi-arc of $D$ satisfies that the pair (orientation, co-orientation) matches the (right-handed) orientation of $\mathbb R^2$.  
The co-orientation is also called the {\it normal} in this paper.  

Note that at a crossing, if the pair of the co-orientation of the over-arc and that of the under-arc matches the (right-handed) orientation of $\mathbb R^2$, then the crossing is positive; otherwise it is  negative. The crossing in (a) of Fig.~\ref{fig-col-0} is positive and that in (b) is negative. 

Among the four quadrants around a crossing $c$, the unique quadrant from which all co-orientations of the two arcs point outward is called the {\it source region} of $c$. The regions labeled by $R$ in (a) and (b) of Fig.~\ref{fig-col-0} are source regions.

\begin{definition}\label{defn-qc-1}
Let $X$ be a biquandle and let $D$ be a marked graph diagram. A {\it biquandle coloring} of $D$ by $X$, or simply {\it biquandle $X$-coloring} of $D$, is a map $\mathcal{C}: S(D) \rightarrow X$ satisfying the following three conditions:
\begin{itemize}
\item [(1)] For each positive crossing $c$, let $s_1, s_2, s_3$ and $s_4$ be the semi-arcs with co-orientations as shown in (a) of Fig.~\ref{fig-col-0}. Then 
$$\mathcal C(s_3)= \lt{\mathcal C(s_1)}{\mathcal C(s_2)}~\text{and}~\mathcal C(s_4)= \ut{\mathcal C(s_2)}{\mathcal C(s_1)}.$$  

\item [(2)] For each negative crossing $c$, let $s_1, s_2, s_3$ and $s_4$ be the semi-arcs with co-orientations as shown in (b) of Fig.~\ref{fig-col-0}. Then 
$$\mathcal C(s_3)= \ut{\mathcal C(s_1)}{\mathcal C(s_2)}~\text{and}~\mathcal C(s_4)= \lt{\mathcal C(s_2)}{\mathcal C(s_1)}.$$  

\item [(3)] For each marked vertex $v$, let $s_1, s_2, s_3$ and $s_4$ be the semi-arcs of $D$ as shown in (c) or (d) of Fig.~\ref{fig-col-0}. 
Then $\mathcal C(s_1)=\mathcal C(s_2)=\mathcal C(s_3)=\mathcal C(s_4)$. 

\begin{figure}[ht]
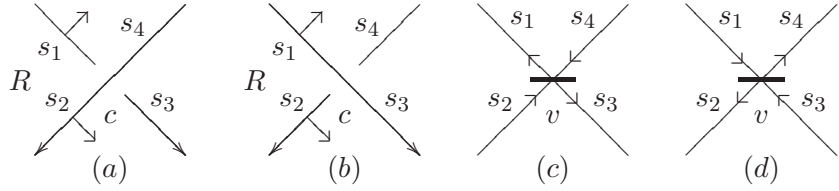

\centerline{
\xy (-10,10);(-2,2) **@{-},(10,-10);(2,-2) **@{-} **@{-} ?<*\dir{<}, 
(10,10);(-10,-10) **@{-} **@{-} ?>*\dir{>},  
(-5,-5);(-2,-8) **@{-},(-2.7,-7.4) *{\lrcorner}, 
(7,-3) *{s_3},
(-6,6);(-3,9) **@{-},(-3.7,7.7) *{\urcorner},
(-7,-3) *{s_2},(0,-5) *{c},
(0,-12)*{(a)}, (-8,3.8) *{s_1}, (3,7) *{s_4}, (-12,0) *{R},
\endxy 
 \qquad
\xy (10,-10);(-10,10) **@{-}  **@{-} ?<*\dir{<}, 
(10,10);(2,2) **@{-},
(-2,-2);(-10,-10) **@{-} **@{-} ?>*\dir{>},  
(-5,-5);(-2,-8) **@{-},(-2.7,-7.4) *{\lrcorner}, 
(7,-3) *{s_3}, (-8,3.8) *{s_1},
(-7,-3) *{s_2},(0,-5) *{c},
(0,-12)*{(b)}, (3,7) *{s_4}, (-12,0) *{R},
(-6,6);(-3,9) **@{-},(-3.7,7.7) *{\urcorner},
\endxy
\qquad
\xy (-10,10);(10,-10) **@{-}, 
(10,10);(-10,-10) **@{-}, 
(3,3.2)*{\llcorner}, 
(-3,-3.4)*{\urcorner}, 
(-2.5,2)*{\ulcorner},
(2.5,-2.4)*{\lrcorner}, 
(3,-0.2);(-3,-0.2) **@{-},
(3,0);(-3,0) **@{-}, 
(3,0.2);(-3,0.2) **@{-}, 
(7,-3) *{s_3},(-4,8) *{s_1},(-7,-3) *{s_2},(4,8) *{s_4},(0,-5) *{v},
(0,-12)*{(c)},
\endxy
 \qquad
\xy (-10,10);(10,-10) **@{-}, 
(10,10);(-10,-10) **@{-},  
(2.5,2)*{\urcorner}, 
(-2.5,-2.2)*{\llcorner}, 
(-3.2,3)*{\lrcorner},
(3,-3.4)*{\ulcorner},
(3,-0.2);(-3,-0.2) **@{-},
(3,0);(-3,0) **@{-}, 
(3,0.2);(-3,0.2) **@{-}, 
(7,-3) *{s_3},(-4,8) *{s_1},(-7,-3) *{s_2},(4,8) *{s_4},(0,-5) *{v},
(0,-12)*{(d)},
\endxy 
}
\vskip.1cm
\caption{Labels of semi-arcs}\label{fig-col-0}
\end{figure}
\end{itemize}
   
In cases (1) and (2), $s_1$ (or $s_2$) is called the {\it source semi-arc} and $s_3$ (or $s_4$) is called the {\it target semi-arc} at $c$. The biquandle element $\mathcal{C}(s_i)$ is called a {\it color} of the semi-arc $s_i$. We denote by ${\rm Col}^B_X(D)$ the set of all biquandle $X$-colorings of $D$. 
\end{definition}

Let $D$ be a marked graph diagram and let $\widehat D$ be the diagram obtained from $D$ by removing orientations of arcs and all crossings of $D$ as depicted in Fig.~\ref{fig-cr-remov}. (We do not remove marked vertices of $D$.) 
\begin{figure}
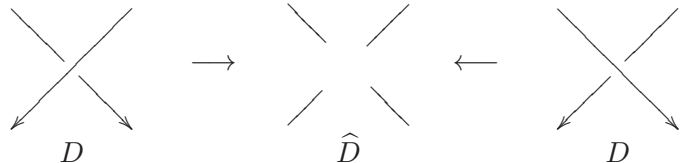

\centerline{ \xy (10,10);(26,26) **@{-} ?<*\dir{<},
(26,10);(19,17) **@{-}  ?<*\dir{<}, (17,19);(10,26) **@{-},
(18,7)*{D}, 
\endxy
 \qquad
 \xy (20.5,20.5);(26,26) **@{-},% ?>*\dir{>}, 
 (10,10);(14.5,14.5) **@{-}, 
(26,10);(21,15) **@{-}, (15,21);(10,26) **@{-},% ?>*\dir{>},
(0,18)*{\longrightarrow}, (35,18)*{\longleftarrow}, (18,7)*{\widehat D},
\endxy
\qquad
\xy (10,26);(26,10) **@{-} ?>*\dir{>}, 
(26,26);(19,19) **@{-}, 
(17,17);(10,10) **@{-} ?>*\dir{>}, (18,7)*{D},
\endxy
}
\caption{Removing a crossing}\label{fig-cr-remov}
\end{figure}
We label the connected components of $\widehat D$ as $b_{1},\ldots,b_{n}$. Let $C(D)=\{c_{1},\ldots,c_{m}\}$ be the set of all crossings of $D$ and let $\epsilon_i (1 \leq i \leq m)$ be the sign of the crossing $c_i$. For each crossing $c_i \in C(D)$, if the crossing $c_i$ has four adjacent connected components (they may not be distinct), say $b_{i_1}, b_{i_2}, b_{i_3}, b_{i_4}$, we produce two relations $r_{\epsilon_i}^{1}(c_{i})$ and $r_{\epsilon_i}^{2}(c_{i})$ depending on the crossing sign $\epsilon_i=+$ or $-$ as shown in Fig.~\ref{fig-qd-rels}. 

\begin{figure}[th]
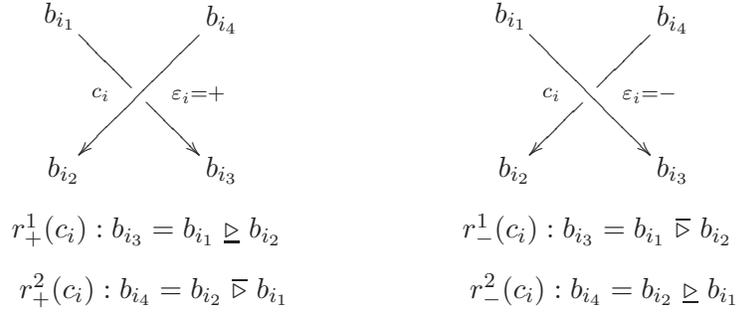

\centerline{ \xy (10,10);(26,26) **@{-} ?<*\dir{<},
(26,10);(19,17) **@{-} ?<*\dir{<}, (17,19);(10,26) **@{-},
(8,8)*{b_{i_2}}, (29,8)*{b_{i_3}}, (7.5,28.5)*{b_{i_1}}, (29,28)*{b_{i_4}},
(26,18)*{_{\varepsilon_i=+}}, 
(13,18)*{_{c_i}},
(19,0)*{r_{+}^1(c_i): b_{i_3}=\lt{b_{i_1}}{b_{i_2}}},
(20,-8)*{r_{+}^2(c_i): b_{i_4}=\ut{b_{i_2}}{b_{i_1}}},
\endxy
 \qquad\qquad\qquad
\xy (10,26);(26,10) **@{-} ?>*\dir{>}, 
(26,26);(19,19) **@{-}, (17,17);(10,10) **@{-} ?>*\dir{>},
(8,8)*{b_{i_2}}, (29,8)*{b_{i_3}}, (7.5,28.5)*{b_{i_1}}, (29,28)*{b_{i_4}},
(26,18)*{_{\varepsilon_i=-}}, 
(13,18)*{_{c_i}},
(19,0)*{r_{-}^1(c_i): b_{i_3}=\ut{b_{i_1}}{b_{i_2}}},
(20,-8)*{r_{-}^2(c_i): b_{i_4}=\lt{b_{i_2}}{b_{i_1}}},
\endxy
}
 \vspace*{8pt} \caption{Relations at a crossing}
 \label{fig-qd-rels}
\end{figure} 

Then the biquandle $BQ(D)$ associated with $D$ is defined by the biquandle with a presentation:
\begin{equation}\label{pr-biqd}
P_D=<b_{1},\hdots,b_{n}~|~r_{\epsilon_1}^{1}(c_{1}), r_{\epsilon_1}^{2}(c_{1}),\ldots,r_{\epsilon_m}^{1}(c_{m}), r_{\epsilon_m}^{2}(c_{m})>.
\end{equation}
It is noted that if $D$ and $D'$ are two admissible marked graph diagrams presenting the same surface-link $\mathcal L$, then $BQ(D)$ is isomorphic to $BQ(D')$ and hence $BQ(D)$ is an invariant of the surface-link $\mathcal L$. The biquandle $BQ(D)$ with a presentation in (\ref{pr-biqd}) is called the {\it fundamental biquandle} of $\mathcal L$ and denoted by $BQ(\mathcal L)$ (cf. \cite{As},\cite{Ca}). If $D$ is a classical link diagram, then the biquandle $BQ(D)$ with a presentation in (\ref{pr-biqd}) is the fundamental biquandle of the classical link $L$ in $\mathbb R^3$ presented by $D$.

\bigskip

Now let $\mathcal B$ be a broken surface diagram of a surface-link $\mathcal L$ in $\mathbb R^4$ and let $S(\mathcal B)$ be the set of the semi-sheets in $\mathcal B$. For a given finite biquandle $X$, we define a {\it biquandle coloring} of $\mathcal B$ by $X$ or a {\it  biquandle $X$-coloring} of $\mathcal B$ to be a function $\mathcal C:S(\mathcal B) \rightarrow X$ satisfying the following condition at each double point curve: At a double point curve, two coordinate planes intersect locally and one is the under-sheet and the other is the over-sheet. The under-sheet (resp. over-sheet) is broken into two semi-sheets, say $u_1$ and $u_2$ (resp. $o_1$ and $o_2$). A normal of the under-sheet (resp. over-sheet) points to one of the components, say $o_2$ (resp. $u_2$). If $\mathcal C(u_1)=a$ and $\mathcal C(o_1)=b,$ then we require that $\mathcal C(u_2)=\lt{a}{b}$ and $\mathcal C(o_2)=\ut{b}{a}$ (see Fig.~\ref{fig-dpc}). The biquandle element $\mathcal C(s)$ assigned to a semi-sheet $s$ by a biquandle coloring is called a {\it color} of $s.$ Using biquandle axioms, it is easily checked that the above condition is compatible at each triple point of $\mathcal B$ as illustrated in Fig.~\ref{fig-dpc} (right). We denote by ${\rm Col}^B_X(\mathcal B)$ the set of all biquandle colorings of $\mathcal B$ by $X$. 

\begin{figure}[ht]
\begin{center}
\resizebox{0.8\textwidth}{!}{%
  \includegraphics{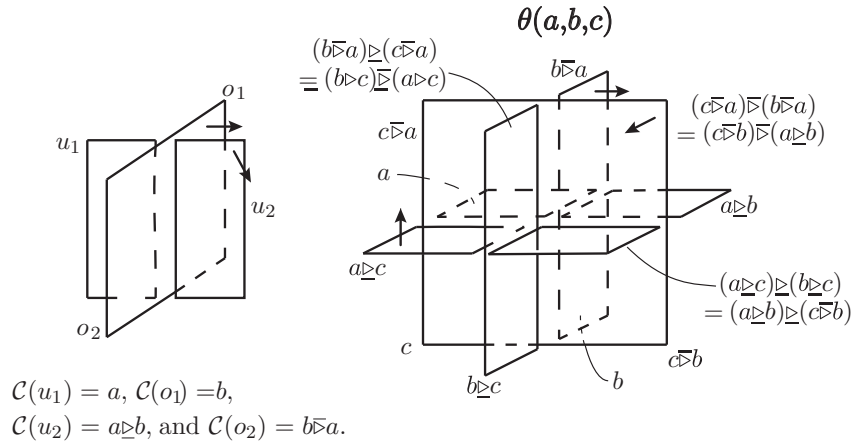}}
\caption{Biquandle colors at a double point curve and a triple point}\label{fig-dpc}
\end{center}
\end{figure}

\begin{theorem}[\cite{Ca}]\label{thm-bsd-cnbr}
Let $\mathcal L$ be a surface-link and let $\mathcal B$ and $\mathcal B'$ be two broken surface diagrams of $\mathcal L$. Then for any finite biquandle $X$, there is a one-to-one correspondence between ${\rm Col}^{\rm B}_{X}(\mathcal B)$ and ${\rm Col}^{\rm B}_{X}(\mathcal B')$. 
\end{theorem}

\begin{proof}
By Theorem \ref{thm-YosM}, it suffices to verify the assertion for the case that $\mathcal B'$ is obtained from $\mathcal B$ by a single Roseman move. Let $E$ be an open $3$-disk in $\mathbb R^3$ where a Roseman move under consideration is applied. Then $\mathcal B\cap (\mathbb R^3 - E) = \mathcal B'\cap (\mathbb R^3 -E)$. Now let $\mathcal C$ be a biquandle $X$-coloring of $\mathcal B$. Using biquandle axioms of Definition \ref{defn-biqdle} and Fig.~\ref{fig-dpc}, for each Roseman move, it is seen that the restriction of $\mathcal C$ to $\mathcal B\cap (\mathbb R^3 - E) (=\mathcal B'\cap (\mathbb R^3 -E))$ can be extended to a unique biquandle $X$-coloring $\mathcal C'$ of $\mathcal B'$ and conversely the restriction of the unique biquandle $X$-coloring $\mathcal C'$ to $\mathcal B'\cap (\mathbb R^3 -E)$ is extended to the biquandle $X$-coloring $\mathcal C$ of $\mathcal B$. \end{proof}

\begin{theorem}[\cite{As}]\label{thm-col}
Let $\mathcal L$ be a surface-link and let $D$ and $\mathcal B$ be a marked graph diagram and a broken surface diagram presenting $\mathcal L$, respectively. Then there is a one-to-one correspondence between ${\rm Col}^B_X(D)$ and ${\rm Col}^B_X(\mathcal B)$. Consequently, $\sharp{\rm Col}^B_X(D)=\sharp{\rm Col}^B_X(\mathcal B)$.
\end{theorem}

\begin{proof}
Let $D$ be a marked graph diagram presenting a surface-link $\mathcal L$ in $\mathbb R^4$. It is direct by definition that there is a one-to-one correspondence between ${\rm Col}^B_X(D)$ and ${\rm Hom}(BQ(D),X)$. 
By Theorem \ref{thm-bsd-cnbr}, we may assume that $\mathcal B$ is a broken surface diagram $\mathcal B(D)$ associated with $D$ defined in Section \ref{sect-mgd}. Then by the same argument as in \cite{As} we see that there is a natural isomorphism from $BQ(D)$ to the fundamental quandle $BQ(\mathcal B)$ of $\mathcal B$. Since ${\rm Col}^B_X(D)$ is identified with ${\rm Hom}(BQ(D),X)$ and ${\rm Col}^B_X(\mathcal B)$ is identified with ${\rm Hom}(BQ(\mathcal B),X)$, there is a bijection from ${\rm Col}^B_X(D)$ to ${\rm Col}^B_X(\mathcal B)$. 
\end{proof}

We call the cardinal number $\sharp{\rm Col}^B_X(\mathcal L)$ the {\it biquandle $X$-coloring number} of $\mathcal L$ (see Example \ref{examp-1}). In particular, if $D$ is a classical link diagram, i.e., a marked graph diagram without marked vertices, then the biquandle $X$-coloring number $\sharp{\rm Col}^B_X(D)$ is an invariant of the link $L$ in $\mathbb R^3$ presented by $D$. 

%%%

\section{Biquandle (co)homology groups}\label{sect-bqht}

Let $X$ be a biquandle. For each positive integer $n$, let $C_n^{\rm BR}(X)$ be the free abelian group generated by $n$-tuples $(x_1,\ldots,x_n)$ of elements of $X$. We assume that $C_n^{\rm BR}(X)=\{0\}$ for all $n\leq 0$. Define a homomorphism $\partial_n:C_n^{\rm BR}(X)\rightarrow C_{n-1}^{\rm BR}(X)$ by
\begin{equation*}
\begin{split}
\partial_n(x_1,\ldots,x_n)&=\sum_{i=1}^{n}(-1)^{i}[(x_1,\ldots,x_{i-1},x_{i+1},\ldots,x_n)\\
&-(\lt{x_1}{x_i},\ldots,\lt{x_{i-1}}{x_i},\ut{x_{i+1}}{x_i},\ldots,\ut{x_n}{x_i})]
\end{split}
\end{equation*}
for $n\geq2$ and $\partial_n=0$ for $n\leq1.$ It is verified that for each integer $n$, $\partial_{n-1}\partial_n=0.$ Therefore $C_\ast^{\rm BR}(X)= \{C_n^{\rm BR}(X),\partial_n\}$ is a chain complex.

Let $C_n^{\rm BD}(X)$ be the subset of $C_n^{\rm BR}(X)$ generated by $n$-tuples $(x_1,\ldots,x_n)$ with $x_i=x_{i+1}$ for some $i\in\{1,\ldots,n-1\}$ if $n\geq2$; otherwise, let $C_n^{\rm BD}(X)=0.$ Then $\partial_n(C_n^{\rm BD}(X))\subset C_{n-1}^{\rm BD}(X)$ and $C_\ast^{\rm BD}(X)=\{C_n^{\rm BD}(X),\partial_n\}$ is a subcomplex of $C_\ast^{\rm BR}(X).$ Put $C_n^{\rm BQ}(X)=C_n^{\rm BR}(X)/C_n^{\rm BD}(X)$ and $C_\ast^{\rm BQ}(X)=\{C_n^{\rm BQ}(X),\partial_n\}$.

For an abelian group $A$, we define the biquandle chain and cochain complexes
\begin{align*}
&C_\ast^{\rm W}(X;A)=C_\ast^{\rm W}(X)\otimes A,& &\partial=\partial\otimes \text{id},
\\
&C^\ast_{\rm W}(X;A)=\text{Hom}(C_\ast^{\rm W}(X), A),& &\delta=\text{Hom}(\partial,\text{id})
\end{align*}
in the usual way, where ${\rm W=BR, BD, BQ}$.

\begin{definition}
The $n$th {\it birack homology group} and the $n$th {\it birack cohomology group} of a birack (biquandle) $X$ with coefficient group $A$ are defined by
\begin{align*}
H_n^{\rm BR}(X;A)=H_n(C_\ast^{\rm BR}(X;A)), H_{\rm BR}^n(X;A)=H^n(C^\ast_{\rm BR}(X;A)).
\end{align*}
The $n$th {\it degeneration homology group} and the $n$th {\it degeneration  cohomology group} of a biquandle $X$ with coefficient group $A$ are defined by
\begin{align*}
H_n^{\rm BD}(X;A)=H_n(C_\ast^{\rm BD}(X;A)), H_{\rm BD}^n(X;A)=H^n(C^\ast_{\rm BD}(X;A)).
\end{align*}
The $n$th {\it biquandle homology group} and the $n$th {\it biquandle  cohomology group} of a biquandle $X$ with coefficient group $A$ are defined by
\begin{align*}
H_n^{\rm BQ}(X;A)=H_n(C_\ast^{\rm BQ}(X;A)), H_{\rm BQ}^n(X;A)=H^n(C^\ast_{\rm BQ}(X;A)).
\end{align*}
\end{definition}

The $n$-cycle group and $n$-boundary group (resp. $n$-cocycle group and $n$-coboundary group) are denoted by $Z_n^{\rm BQ}(X;A)$ and $B^{\rm BQ}_n(X;A)$ (resp. $Z_{\rm BQ}^n(X;A)$ and $B^n_{\rm BQ}(X;A)$). Then
$$H_n^{\rm BQ}(X;A)=Z_n^{\rm BQ}(X;A)/B_n^{\rm BQ}(X;A), H_{\rm BQ}^n(X;A)=Z_{\rm BQ}^n(X;A)/B_{\rm BQ}^n(X;A).$$
We will omit the coefficient group $A$ if $A=\mathbb Z$ as usual.
The biquandle homology and cohomology groups are also known as the {\it Yang-Baxter homology} and {\it cohomology groups} (cf. \cite{CES}). 

\begin{lemma}\label{rmk-bq2cocy} 
A homomorphism $\phi: C^{\rm BR}_2(X) \to A$ is a $2$-cocycle of the biquandle cochain complex $C^\ast_{\rm BQ}(X;A)$ if and only if $\phi$ satisfies the following two conditions:
\begin{itemize}
\item [(i)] $\phi(a,a)=0$ for all $a \in X$.
\item [(ii)] $\phi(b,c)+\phi(a,b)+\phi(\lt{a}{b},\ut{c}{b})=\phi(a,c)+\phi(\ut{b}{a},\ut{c}{a})+\phi(\lt{a}{c},\lt{b}{c})$ for all $a, b, c \in X$. 
\end{itemize}
\end{lemma}

\begin{proof}
It follows from the definition by a direct calculation. 
\end{proof}

The two conditions (i) and (ii) are called the {\it biquandle 2-cocycle condition}. A homomorphism $\phi: C^{\rm BR}_2(X) \to A$ or a map $\phi : X \times X \to A$ satisfying the biquandle 2-cocycle condition is called a {\it biquandle $2$-cocycle}.

\begin{lemma}\label{rmk-bq3cocy}
A homomorphism $\theta: C^{\rm BR}_3(X) \to A$ is a 3-cocycle of the biquandle cochain complex $C^\ast_{\rm BQ}(X;A)$ if and only if $\theta$ satisfies the following two conditions:
\begin{itemize}
\item [(i)] $\theta(a,a,b)=0$ and $\theta(a,b,b)=0$ for all $a, b \in X$.
\item [(ii)] $\theta(b,c,d)+\theta(a,b,d)+\theta(\lt{a}{b},\ut{c}{b},\ut{d}{b})+\theta(\lt{a}{d},\lt{b}{d},\lt{c}{d})$
$=\theta(a,c,d)+\theta(a,b,c)+\theta(\ut{b}{a},\ut{c}{a},\ut{d}{a})+\theta(\lt{a}{c},\lt{b}{c},\ut{d}{c})$ for all $a, b, c, d \in X$.
\end{itemize}
\end{lemma}

\begin{proof}
It follows from the definition by a direct calculation. 
\end{proof}

The two conditions (i) and (ii) are called the {\it biquandle 3-cocycle condition}. A homomorphism $\theta : C^{\rm BR}_3(X) \to A$ or a map $\theta : X \times X \times X \to A$ satisfying the biquandle 3-cocycle condition is called a {\it biquandle $3$-cocycle}.

%%%

\section{Biquandle cocycle invariants of surface-links}
\label{sect-cisl}

 By using the cohomology theory of quandles, the {\it quandle cocycle invariants} 
are defined for classical links and surface-links via link diagrams and broken surface  diagrams (\cite{CJKLS}). These invariants are defined as state-sums over all quandle colorings of arcs and sheets by use of Boltzman weights that are evaluations of a fixed cocycle at crossings and triple points in link diagrams and broken surface diagrams, respectively. In \cite{CES}, J. S. Carter, M. Elhamdadi and M. Saito introduced a homology theory for the set-theoretic Yang-Baxter equations and used cocycles to define invariants of (virtual) links via colorings of (virtual) link diagrams by biquandles and a state-sum formulation. In this section, we first recall the biquandle cocycle invariants of links and then develop biquandle cocycle invariants of surface-links which is a generalization of quandle cocycle invariants of surface-links. We begin with introducing biquandle cocycle invariants of links in the terminologies of this paper.

Let $X$ be a finite biquandle and let $A$ be an abelian group written multiplicatively.
Let $\phi$ be a biquandle $2$-cocycle. Let $D$ be a link diagram and let $\mathcal C$ be a biquandle $X$-coloring of $D$. A ({\it Boltzman}) {\it weight} $W^B_\phi(c,\mathcal C)$ (associated with $\phi$) at a crossing $c$ of $D$ is defined as follows. Let $u_1$ (resp. $o_1$) be the semi-arc in the under-arc (resp. over-arc) that intersects with the source region of the crossing $c$. Let $a=\mathcal C(u_1)$ and  $b=\mathcal C(o_1)$.  Then we define $W^B_\phi(c,\mathcal C)=\phi(a,b)^{\epsilon(c)} \in A$, where $\epsilon(c)=1$ or $-1$ if the sign of $c$ is positive or negative, respectively. The {\it state-sum} or {\it partition function} of $D$ (associated with $\phi$) is defined by the expression $$\Phi^B_\phi(D;A)=\sum_{\mathcal C\in {\rm Col}^B_X(D)}\prod_{c\in C(D)}W^B_\phi(c,\mathcal C),$$ where the value of the state-sum $\Phi^B_\phi(D;A)$ is in the group ring $\mathbb Z[A].$ 

\begin{theorem}[\cite{CES}]\label{thm-inv}
Let $D$ be a diagram of a link $L$ in $\mathbb R^3$ and
let $\phi$ be a biquandle $2$-cocycle. The state-sum $\Phi^B_\phi(D;A)$ of $D$ (associated with $\phi$) is an invariant of $L$. (It is denoted by $\Phi^B_\phi(L;A)$.)  \qed
\end{theorem}

The state-sum invariant $\Phi^B_\phi(L;A)$ is also called the {\it biquandle cocycle invariant} of $L$.  
We say that the state-sum invariant $\Phi^B_\phi(L;A)$ of a knot or link $L$ is said to be {\it trivial} if it is an integer. In this case, the integer is equal to the number of biquandle $X$-colorings of a diagram $D$ of $L$, i.e., $\Phi^B_\phi(L;A)=\sharp{\rm Col}^B_X(D)$. 

The following proposition shows that 
the state-sum invariant $\Phi^B_\phi(L;A)$ depends on the cohomology class of a biquandle 2-cocycle $\phi$.  
It is verified by a similar argument as in \cite[Proposition 4.5]{CJKLS} and we omit the proof.  

\begin{proposition}\label{prop-cleq}
Let $\phi,\phi'\in Z_{\rm BQ}^2(X;A)$ be biquandle $2$-cocycles. If $\Phi^B_\phi$ and $\Phi^B_{\phi'}$ denote the state-sum invariants defined from cohomologous biquandle $2$-cocycles $\phi$ and $\phi'$ (so that $\phi=\phi'\delta\psi$ for some biquandle $1$-cochain $\psi$), then $\Phi^B_\phi=\Phi^B_{\phi'}$ (so that $\Phi^B_\phi(L;A)=\Phi^B_{\phi'}(L;A)$ for any oriented link $L$). In particular, the state-sum is trivial if the $2$-cocycle used for the Boltzman weight is a coboundary. \qed
\end{proposition}
   
Now let $\mathcal L$ be a surface-link in $\mathbb R^4$ and let $\mathcal B$ be a broken surface diagram of $\mathcal L$. Let $\tau$ be a triple point of $\mathcal B$ intersecting three sheets that have relative positions top, middle, and bottom with respect to the projection direction of $p:\mathbb R^4\rightarrow \mathbb R^3.$ The {\it sign} of the triple point $\tau$ is defined to be {\it positive} if the normals of top, middle, bottom sheets in this order match the orientation of the $3$-space $\mathbb R^3$. Otherwise, the {\it sign} of $\tau$ is defined to be {\it negative}. We use the right-handed rule convention for the orientation of $\mathbb R^3$.

Fix a finite biquandle $X$, an abelian group $A$ written multiplicatively, and a biquandle $3$-cocycle $\theta$. Let $R$ be the source region of a triple point $\tau$ in $\mathcal B$, that is, the octant from which all normal vectors of the three sheets point outwards. For a given biquandle $X$-coloring $\mathcal C$ of $\mathcal B$, let $a,b$ and $c$ be the colors of the bottom, middle and top sheets, respectively, that bound the source region $R$. Set $\epsilon(\tau)=1$ or $-1$ according as $\tau$ is positive or negative, respectively. The {\it (Boltzman) weight} $W^B_\theta(\tau,\mathcal C)$ at $\tau$ with respect to $\mathcal C$ is defined by $$W^B_\theta(\tau,\mathcal C)=\theta(a,b,c)^{\epsilon(\tau)}\in A.$$ For example, see Fig.~\ref{fig-dpc}, where $\epsilon(\tau)=1$.

\begin{figure}[ht]
\begin{center}
\resizebox{0.65\textwidth}{!}{%
  \includegraphics{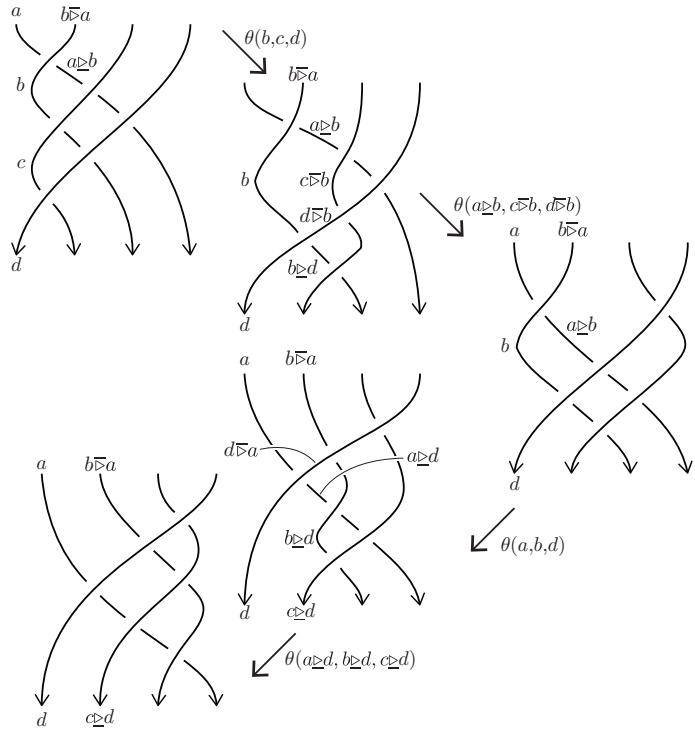}}
\caption{The tetrahedral move and a biquandle 3-cocycle condition (ii), LHS}\label{fig-slcol1}
\end{center}
\end{figure}

\begin{definition}\label{def-slci}
Let $\mathcal B$ be a broken surface diagram of a surface-link and let $\theta$ be a biquandle $3$-cocycle. The {\it state-sum} or {\it partition function} of $\mathcal B$ (associated with $\theta$) is defined to be the sum 
$$\Phi^B_\theta(\mathcal B; A)=\sum_{\mathcal C \in {\rm Col}^B_X(\mathcal B)}\prod_{\tau\in T(\mathcal B)}W^B_\theta(\tau,\mathcal C),$$ 
where $T(\mathcal B)$ denotes the set of all triple points of $\mathcal B$ and the state-sum $\Phi^B_\theta(\mathcal B; A)$ is an element of the group ring $\mathbb Z[A].$ 
\end{definition}

\begin{theorem}\label{thm-bqc-inv-bsd}
Let $\mathcal L$ be a surface-link and let $\mathcal B$ be a broken surface diagram of $\mathcal L$. For a given biquandle $3$-cocycle $\theta$, the state-sum $\Phi^B_\theta(\mathcal B; A)$ is an invariant of $\mathcal L$. 
(It is denoted by $\Phi^B_\theta(\mathcal L; A)$.)  
\end{theorem}

\begin{proof}
Suppose that $\mathcal B'$ is a broken surface diagram obtained from $\mathcal B$ by a single Roseman move. For each biquandle $X$-coloring $\mathcal C$ of $\mathcal B$, let $\mathcal C'$ be the corresponding biquandle $X$-coloring of $\mathcal B'$ as in the proof of Theorem \ref{thm-bsd-cnbr}. By Theorems \ref{thm-RosM} and \ref{thm-bsd-cnbr}, it suffices to prove that $\prod_{\tau \in T(\mathcal B)}W^{\rm B}_\theta(\tau,\mathcal C)=\prod_{\tau \in T(\mathcal B')}W^{\rm B}_\theta(\tau,\mathcal C')$. From Fig.~\ref{fig-rose}, it is immediate that the Roseman moves of type 1, 2, 3 and 4 involve no triple points. For the Roseman move of type 5, the product of weights differ by $\theta(x_1,x_1,x_2)^{\pm1}$ or $\theta(x_1,x_2,x_2)^{\pm1}$ for some $x_1,x_2 \in X$. But it follows from Lemma~\ref{rmk-bq3cocy} (i) that $\theta(x_1,x_1,x_2)^{\pm1}=\theta(x_1,x_2,x_2)^{\pm1}=1$ and hence the product of the weights are unchanged. For the Roseman move of type 6, two triple points with the same weights of opposite signs are involved and the product of the weights are canceled. For the Roseman move of type 7, there are four involved triple points before and after the move as illustrated in Figs.~\ref{fig-slcol1} and \ref{fig-slcol2} in motion pictures (the tetrahedral move). From Lemma~\ref{rmk-bq3cocy} (ii), it is seen that the product of the weights are not changed. 
\end{proof}

\begin{figure}[ht]
\begin{center}
\resizebox{0.65\textwidth}{!}{%
  \includegraphics{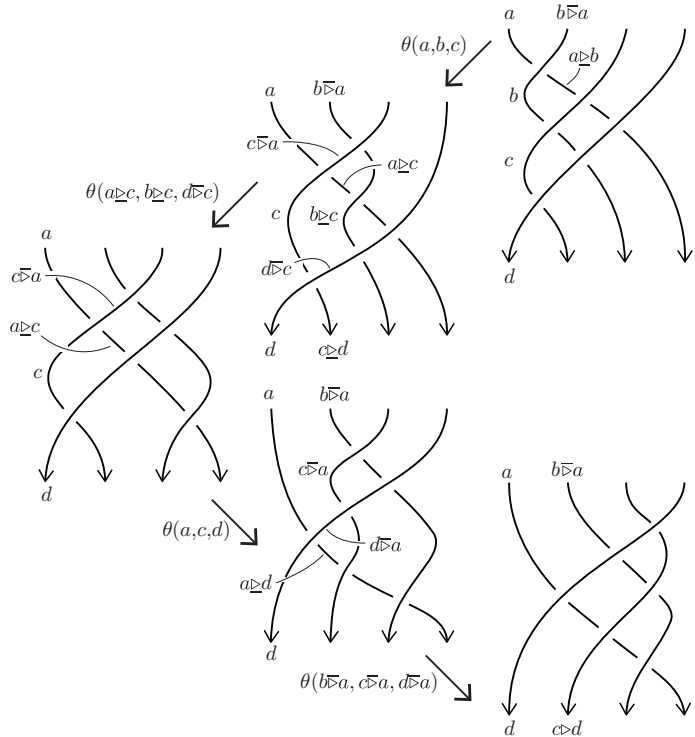}}
\caption{The tetrahedral move and a biquandle 3-cocycle condition (ii), RHS}\label{fig-slcol2}
\end{center}
\end{figure}

In what follows we also call the state-sum invariant $\Phi^B_\theta(\mathcal L; A)$ the {\it biquandle cocycle invariant} of $\mathcal L$.  The following proposition shows that $\Phi^B_\theta(\mathcal B; A)$ depends on the cohomology class of a biquandle 3-cocycle $\theta$ and is proved by a similar argument as in \cite[Proposition 5.7]{CJKLS}.

\begin{proposition}\label{prop-sleq}
Let $\mathcal B$ be a broken surface diagram of a surface-link $\mathcal L$ and let
$\theta$ and $\theta'$ be two cohomologous biquandle $3$-cocycles (so that $\theta=\theta'\delta\phi$ for some biquandle $2$-cochain $\phi$). Then $\Phi^B_\theta(\mathcal B; A)=\Phi^B_{\theta'}(\mathcal B; A)$ and consequently $\Phi^B_\theta(\mathcal L; A)=\Phi^B_{\theta'}(\mathcal L; A)$ for any surface-link $\mathcal L$. In particular, if $\theta$ is a coboundary, then $\Phi^B_\theta(\mathcal L; A)$ is trivial for all surface-link $\mathcal L$, i.e., $\Phi^B_\theta(\mathcal L; A)=\sharp{\rm Col}^B_X(\mathcal B)$.\qed
\end{proposition}

%%%

\section{Biquandle cocycle invariants from marked graph\\ diagrams}\label{sect-qcocm}

In this section we give an interpretation of the biquandle $3$-cocycle invariants of surface-links in terms of marked graph diagrams. Let $D$ be a marked graph diagram and let $D_+$ be the positive resolution of $D.$ Let $D_+=D_1 \rightarrow D_2 \rightarrow\cdots\rightarrow D_r=O$ be a sequence of link diagrams from $D_+$ to a trivial link diagram $O$ related by ambient isotopies of $\mathbb R^2$ and Reidemeister moves. Let $$I^3_+=\{i \mid D_{i}\rightarrow D_{i+1} \text{ is a Reidemeister move } R_3\}.$$ For each $i \in I^3_+$, let $B_{(i)}$ be a 2-disk in $\mathbb R^2$ where the move $D_i\rightarrow D_{i+1}$ is applied. Similarly, let $D_-$ be the negative resolution of $D$ and let $D_-=D_1' \rightarrow D_2' \rightarrow\cdots\rightarrow D_s'=O'$ be a sequence of link diagrams from $D_-$ to a trivial link diagram $O'$ related by ambient isotopies of $\mathbb R^2$ and Reidemeister moves. Let $$I^3_-=\{j \mid D_{j}'\rightarrow D_{j+1}' \text{ is a Reidemeister move } R_3\}.$$  For each $j \in I^3_-$, let $B_{(j)}'$ be a 2-disk in $\mathbb R^2$ where the move $D_{j}'\rightarrow D_{j+1}'$ is applied.

Let $i\in I^3_+$ (resp., $j\in I^3_-$).   We note that there exists a unique region in $D_i$ or $D_{i+1}$ (resp., $D_j'$ or $D_{j+1}'$) facing three semi-arcs (the bottom, middle, top semi-arcs) such that all co-orientations of the semi-arcs point outward as depicted in Fig.~\ref{fig-lab1} where the regions are shaded by the blue color. We call the (blue) region the {\it source region of the stage} $i$ (resp., $j$). We define two sign functions $\epsilon_{tm}$ and $\epsilon_{b}$ from the disjoint union $I^3_+ \amalg {I^3_-}$ to $\{\pm1\}$ as follows: Let $i\in I^3_+$ (or $i\in I^3_-$, resp.) and let $c_i$ be the crossing between the top arc and the two middle arcs in  $D_{i}\cap B_{(i)}$ (or $D_{i}'\cap B_{(i)}'$, resp.) 
and let $n_b$ be the co-orientation of the bottom arc. Define $\epsilon_{tm}(i)$ and $\epsilon_{b}(i)$ for $i\in I^3_+\amalg I^3_-$ by
\begin{align}
&\epsilon_{tm}(i)={\rm sign}(c_i),\label{eq-etmi}\\
&\epsilon_{b}(i)=\begin{cases}
1 \hskip 0.9cm \text{if $n_b$ points from }c_i,\\
-1 \hskip 0.6cm \text{otherwise}.~\text{(cf. Fig.~\ref{fig-signft}.)}
\end{cases}\label{eq-ebi}
\end{align}

\begin{figure}[ht]
\begin{center}
\resizebox{0.50\textwidth}{!}{%
  \includegraphics{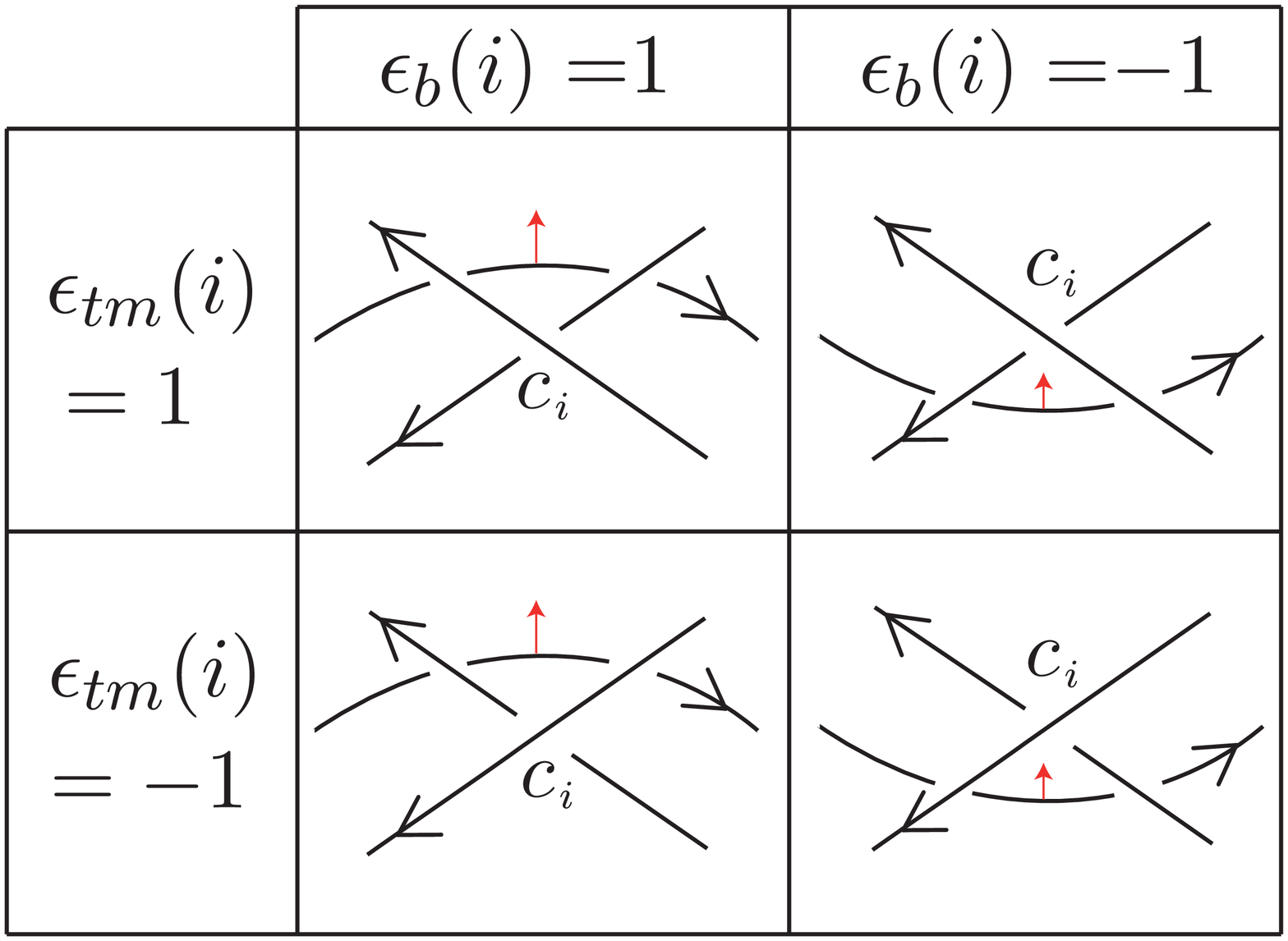} }
\caption{The sign functions $\epsilon_{tm}$ and $\epsilon_{b}$}\label{fig-signft}
\end{center}
\end{figure}

Let $D$ be a marked graph diagram, let $\mathcal{C}:S(D)\rightarrow X$ a biquandle $X$-coloring of $D$, and let $\theta$ be a biquandle $3$-cocycle.
Set $i\in I^3_+ \amalg {I^3_-}$. Let $R$ be the source region of the stage $i$.  The {\it(Boltzman) weight} $W^B_\theta(i,\mathcal{C})$ at $i$ with respect to $\mathcal C$ is defined by $$W^B_\theta(i,\mathcal{C})=\theta(x_1,x_2,x_3)^{\epsilon_{tm}(i)\epsilon_{b}(i)},$$ where $x_1, x_2$ and $x_3$ are the colors of the bottom, middle and top semi-arcs facing the source region $R$ at the stage $i$, respectively, as depicted in Fig.~\ref{fig-lab1}.

 \begin{figure}[ht]
\begin{center}
\resizebox{0.7\textwidth}{!}{%
  \includegraphics{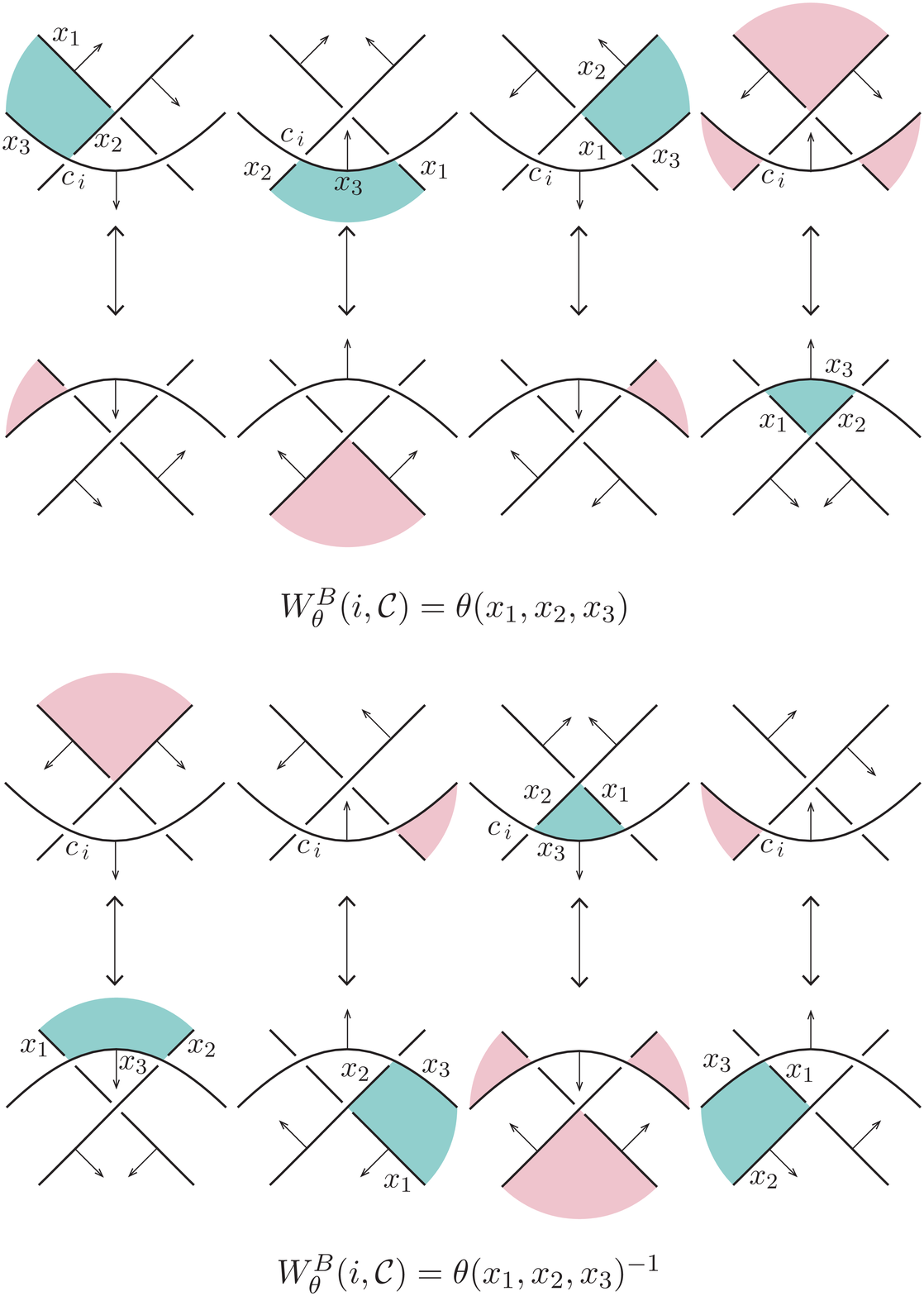} }
\caption{Boltzman weight at $i\in I^3_+ \amalg {I^3_-}$}\label{fig-lab1}
\end{center}
\end{figure}

\begin{definition}\label{def-cocyclem}
Let $D$ be a marked graph diagram of a surface-link $\mathcal L.$ 
For a given biquandle 3-cocycle $\theta$, the {\it state-sum} or {\it partition function} of $D$ (associated with $\theta$) is defined by
\begin{align*}
\Phi^B_\theta(D; A)=\sum_{\mathcal{C}\in {\rm Col}^B_X(D)}\biggl(\prod_{i\in I^3_+ }W^B_\theta(i,\mathcal{C})\prod_{j\in I^3_- }W^B_\theta(j,\mathcal{C})^{-1}\biggr).\end{align*}
\end{definition}

\begin{theorem}\label{thm-cocyclem}
Let $\mathcal L$ be a surface-link and let $D$ be a marked graph diagram of $\mathcal L$. Then for any biquandle $3$-cocycle $\theta,$ $\Phi^B_\theta(\mathcal L; A)=\Phi^B_\theta(D; A)$. 
\end{theorem}

\begin{proof}
Let $\mathcal B=\mathcal{B}(D)$ be a broken surface diagram associated with $D$ defined in Section \ref{sect-mgd}. By Theorem \ref{thm-bqc-inv-bsd}, it is sufficient to prove that $\Phi^B_\theta(D)=\Phi^B_\theta(\mathcal B)$. 
Since there is a natural bijection between  
${\rm Col}^B_X(D)$ and ${\rm Col}^B_X(\mathcal B)$ (see Theorem~\ref{thm-col}), it 
suffices to show the following claim.  

\smallskip
 {\bf Claim}: For each biquandle $X$-coloring $\mathcal C\in {\rm Col}^B_X(D),$
 $$\prod_{i\in I^3_+ }W^B_\theta(i,\mathcal{C})\prod_{j\in I^3_- }W^B_\theta(j,\mathcal{C})^{-1}=\prod_{\tau\in T(\mathcal B)}W^B_\theta(\tau,\tilde{\mathcal C}),$$
 where $\tilde{\mathcal C}\in {\rm Col}^B_X(\mathcal B)$, which corresponds to the biquandle $X$-coloring $\mathcal C$.  
 
 \bigskip
{\it{\bf Proof of Claim.}}    Let $\mathcal B^i_j=\mathcal B\cap (\mathbb R^2\times [t_j',t_i])$ for $i=1, \ldots, r$ and $j=1, \ldots, s.$  Let $\phi:(\mathbb R^2,D_0)\rightarrow (\mathbb R^2\times [t_1',t_1], \mathcal B^1_1)$ be the natural embedding at $t=0$ as shown, for example, in Fig.~\ref{fig-ne}. The vertices of $D_0$ correspond to the saddle points in $\mathcal B^1_1$ and the crossings of  $D_0$ correspond to the intersection of $\mathbb R^2 \times \{0\}$ and the double point curves in $\mathcal B^1_1$. There are no triple points in $\mathcal B^1_1$. 
   
 \begin{figure}[ht]
\begin{center}
\resizebox{0.7\textwidth}{!}{%
  \includegraphics{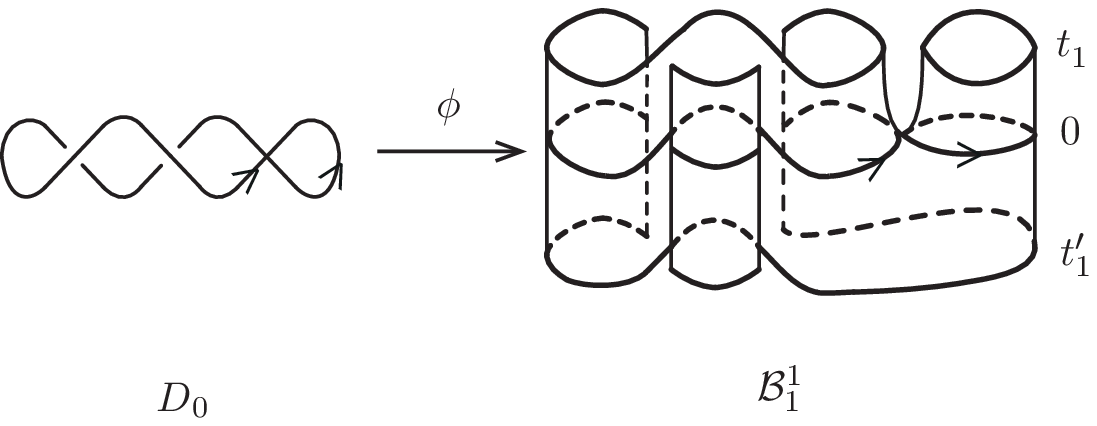} }
\caption{$\phi:(\mathbb R^2,D_0)\rightarrow (\mathbb R^2\times [t_1',t_1], \mathcal B^1_1)$}\label{fig-ne}
\end{center}
\end{figure}

Let $\mathcal B_i=\mathcal B\cap(\mathbb R^2\times [t_{i},t_{i+1}])$ for $i=1,\ldots, r-1$ and 
$\mathcal B_{j}'=\mathcal B\cap(\mathbb R^2\times [t_{j+1}',t_j'])$ for  $j=1,\ldots, s-1.$ Note that 
$T(\mathcal B)=\Big(\overset{r-1}{\underset{i=1}{\cup}}T(\mathcal B_i)\Big)\cup\Big(\overset{s-1}{\underset{j=1}{\cup}}T(\mathcal B_j')\Big)$, where $T(\cdot)$ stands for the set of triple points.  

If the move $D_i\rightarrow D_{i+1}$ is an ambient isotopy of $\mathbb R^2$, then $D_{i}\times [t_{i},t_{i+1}]\cong \mathcal B_i$, and there are no triple points in $\mathcal B_i.$ Suppose that the move $D_i\rightarrow D_{i+1}$ is a  Reidemeister move. Since $D_{i}\setminus B_{(i)}$ and $D_{i+1}\setminus B_{(i)}$ are identical, there are no triple points in $\mathcal B_i\setminus M_{(i)}$ and we have $T(\mathcal B_i)=T(M_{(i)}),$ where $M_{(i)}$ is a subset of $B_{(i)}\times I$ determined by $\pi(M_{(i)}\cap (B_{(i)}\times \{t\}))=\pi(f_t^{(i)}(L(D_i)))\cap B_{(i)}$ for $t\in I$ and a homeomorphism $f_t^{(i)}:\mathbb R^3\rightarrow \mathbb R^3$ satisfying $f_0^{(i)}={\rm id}$ and $f_1^{(i)}(L(D_i))=L(D_{i+1}).$ If the move $D_i\rightarrow D_{i+1}$ is of  type $R_1$ or $R_2$, then there are no triple points in 
 $M_{(i)}$.  See Figs.~\ref{fig-m1} and \ref{fig-m2}. If the move $D_{i}\rightarrow D_{i+1}$ is of type $R_3$, then there is a triple point $\tau_i$ in $M_{(i)}$ as in Fig.~\ref{fig-trip2} and $T(\mathcal B_i)=\{\tau_i\}.$ Then $\overset{r-1}{\underset{i=1}{\cup}}T(\mathcal B_i) = \{\tau_i \mid i \in I^3_+\}.$ 

 \begin{figure}[ht]
\begin{center}
\resizebox{1\textwidth}{!}{%
  \includegraphics{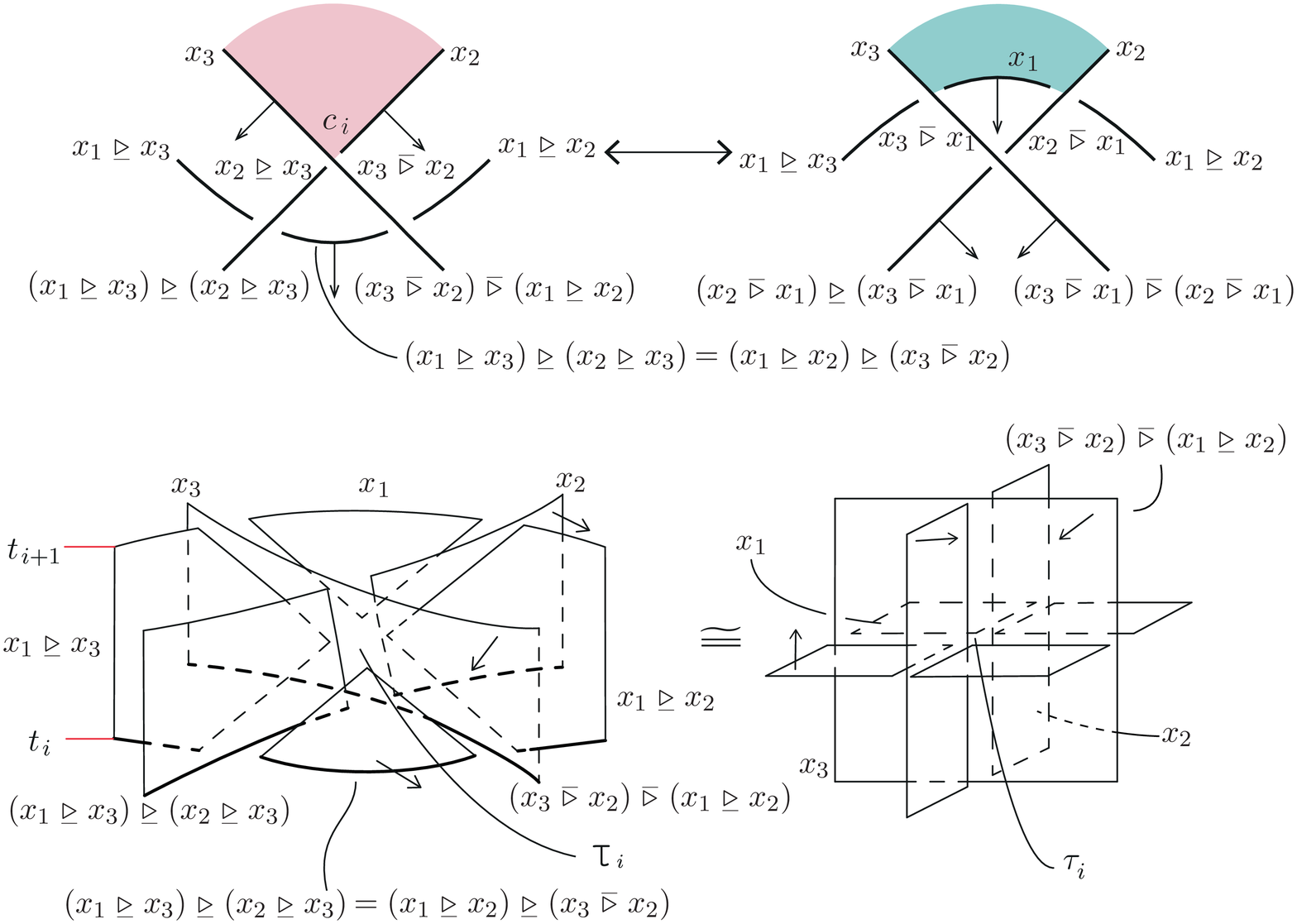} }
\caption{Reidemeister move $R_3$ and corresponding triple point}\label{fig-trip2}
\end{center}
\end{figure}

Similarly, suppose that the move $D_j'\rightarrow D_{j+1}'$ is a  Reidemeister move and $M_{(j)}'$ is a subset of $B_{(j)}'\times I$ determined by $\pi(M_{(j)}'\cap (B_{(j)}'\times \{t\}))=\pi(g_t^{(j)}(L(D_j')))\cap B_{(j)}',$ where $g_t^{(j)}:\mathbb R^3\rightarrow \mathbb R^3$ is a homeomorphism  satisfying $g_0^{(j)}={\rm id}$ and $g_1^{(j)}(L(D_j'))=L(D_{j+1}')$ for $t\in I.$  There is a triple point $\tau_j'\in M_{(j)}'$ for  $j \in I^3_-$.   We have that  $\overset{s-1}{\underset{j=1}{\cup}}T(\mathcal B_j')=\{\tau_j' \mid j \in I^3_-\}.$ Now we have 
   \begin{align*}
   T(\mathcal B)= \{ \tau_i \mid i \in I^3_+\} \cup \{ \tau_j' \mid j \in I^3_-\}. \end{align*}

Let $i\in I^3_+$, i.e.,  $D_i\rightarrow D_{i+1}$ is a Reidemeister move of type $R_3$ and let $\tau_i$ be the corresponding triple point in $M_{(i)}.$ Let $n_b$, $n_m$  and $n_t$ be the co-orientations of the bottom, the middle and the top arcs of $D_i$ in $B_{(i)}$, respectively. By an ambient isotopy, we deform $M_{(i)}$ in $B_{(i)} \times I$ to the standard form of the neighborhood of the triple point $\tau_i$ as in Fig.~\ref{fig-trip2}. Let $\bar{n}_b,$ $\bar{n}_m,$ and $\bar{n}_t$ be the normal vectors corresponding to $n_b,$ $n_m,$ and $n_t$, respectively. Without loss of generality, we may assume $\bar{n}_t= {\bf e}_1$, $\bar{n}_m= \epsilon \/ {\bf e}_2$ and $\bar{n}_b= \epsilon' \/ {\bf e}_3$ for some $\epsilon,\epsilon'\in\{1,-1\}$, where ${\bf e}_1 = (1,0,0)$, ${\bf e}_2 = (0,1,0)$ and ${\bf e}_3 = (0,0,1)$.  
See Fig.~\ref{fig-trip2}. Let $c_i$ be the crossing between the top and the middle arcs in $B_{(i)}$. It is clear from Fig.~\ref{fig-trip2} that $\epsilon={\rm sign}(c_i)$. By (\ref{eq-etmi}), $\epsilon={\rm sign}(c_i)=\epsilon_{tm}(i).$ Hence $\bar{n}_m= \epsilon_{tm}(i) \/ {\bf e}_2$. The sign $\epsilon'$ depends on the co-orientation $n_b$ of the bottom arc. If $n_b$ points from $c_i$, then $\epsilon'=1$. If $n_b$ points toward $c_i,$ then $\epsilon'=-1.$ So, by (\ref{eq-ebi}), $\epsilon'=\epsilon_b(i)$ and hence $\bar{n}_b= \epsilon_b(i) \/ {\bf e}_3$. On the other hand, by definition, the sign $\epsilon(\tau_i)$ of the triple point $\tau_i$ is positive if the co-orientations of the top, the middle and the bottom sheets in this order match the given (right-handed) orientation of $\mathbb{R}^3$. Otherwise, the sign $\epsilon(\tau_i)$ is negative. This gives
\begin{align*}
\epsilon(\tau_i)=\begin{cases}
1 \hskip 1.3cm \text{if } (\bar{n}_t,\bar{n}_m,\bar{n}_b)\in A,\\
-1 \hskip 1cm \text{if }(\bar{n}_t,\bar{n}_m,\bar{n}_b)\in B,\\
\end{cases}
\end{align*}
where 
$A=\{( {\bf e}_1, {\bf e}_2, {\bf e}_3 ), ( {\bf e}_1, -{\bf e}_2, -{\bf e}_3) \}$ and 
$B=\{( {\bf e}_1, -{\bf e}_2, {\bf e}_3 ), ( {\bf e}_1, {\bf e}_2, -{\bf e}_3 ) \}.$  Therefore, for each $i\in I^3_+,$ 
\begin{align}
\label{eq-taui}\epsilon(\tau_i)=\epsilon_{tm}(i)\epsilon_b(i).
\end{align}

Let $j\in I^3_-$, i.e., $D_j'\rightarrow D_{j+1}'$ is a Reidemeister move of type $R_3$. Let $\tau_j'$ be the corresponding triple point in $M_{(j)}'$.  
 Let $n_b$, $n_m$  and $n_t$ be the co-orientations of the bottom, the middle and the top arcs of $D_j'$ in $B_{(j)}'$, respectively. By an ambient isotopy, we deform $M_{(j)}'$ to the standard form of the neighborhood of the triple point $\tau_j'.$ Let $\bar{n}_b,$ $\bar{n}_m,$ and $\bar{n}_t$ be the co-orientations corresponding to $n_b,$ $n_m,$ and $n_t$, respectively. Without loss of generality, we may assume $\bar{n}_t= {\bf e}_1$,  $\bar{n}_m=  \epsilon \/ {\bf e}_2$ and $\bar{n}_b= \epsilon' \/ {\bf e}_3$ for some $\epsilon,\epsilon' \in\{1,-1\}$. Let $c_j$ be the crossing between the top and the middle arcs in $B_{(j)}'$. It is easily seen that $\epsilon={\rm sign}(c_j)$ (cf. Fig.~\ref{fig-trip2}). By (\ref{eq-etmi}), $\epsilon={\rm sign}(c_j)=\epsilon_{tm}(j).$ Hence $\bar{n}_m= \epsilon_{tm}(j) \/ {\bf e}_2$. The sign $\epsilon'$ depends on the co-orientation $n_b$ of the bottom arc. If $n_b$ points from $c_j$, then $\epsilon'=-1$. If $n_b$ points toward $c_j,$ then $\epsilon'=1.$ So, by (\ref{eq-ebi}), $\epsilon'=-\epsilon_b(j)$ and hence $\bar{n}_b= -\epsilon_b(j)\/ {\bf e}_3$. On the other hand, by definition, $\epsilon(\tau_j')=1$ if $\bar{n}_t, \bar{n}_m,$ and $\bar{n}_b$ in this order match the given (right-handed) orientation of $\mathbb{R}^3$. Otherwise, $\epsilon(\tau_j')=-1$. This gives 
\begin{align*}
\epsilon(\tau_j')=\begin{cases}
1 \hskip 1.3cm \text{if } (\bar{n}_t,\bar{n}_m,\bar{n}_b)\in B,\\
-1 \hskip 1cm \text{if }(\bar{n}_t,\bar{n}_m,\bar{n}_b)\in A.\\
\end{cases}
\end{align*} Therefore, for each $j\in I^3_-,$ 
\begin{align}
\label{eq-tauj}\epsilon(\tau_j')=-\epsilon_{tm}(j)\epsilon_b(j).
\end{align}

Now we will show that for each $i \in I^3_+$, $W^B_\theta(i,\mathcal C)=W^B_\theta(\tau_i,\mathcal C)$ and that for each $j \in I^3_-$, $W^B_\theta(j,\mathcal C)=W^B_\theta(\tau_j',\mathcal C)^{-1}$. Let $i\in I^3_+$ (resp. $j \in I^3_-$). Let $R$ be the source region of the stage $i$ (resp. $j$) facing the bottom, middle, top semi-arcs with colors $x_1, x_2, x_3$, respectively, by the coloring $\mathcal C$ as depicted in Fig.~\ref{fig-lab1}. For simplicity, we denote $M_{(i)}$ (resp. $M_{(j)}'$) by $M$ and $[t_i,t_{i+1}]$ (resp. $[t_{j+1}',t_j']$) by $I$ in this proof below. The top (the middle) sheet in $M$ corresponds to the top (the middle) arc times $I$. As shown, for example, in Fig.~\ref{fig-trip2}, $R\times I$ is divided into two ($3$-dimensional) regions by the bottom sheet and one of them is the source region, say $\mathcal R$, of the corresponding triple point $\tau_i$ (resp. $\tau_j'$). The colors $x_1, x_2$ and $x_3$ of the bottom, the middle and the top arc facing the source region $R$ of the stage $i$ (resp. $j$) are the colors of the bottom, the middle and the top sheets facing $\mathcal R$. From (\ref{eq-taui}) and (\ref{eq-tauj}), we have 
$W^B_\theta(i,\mathcal C)=\theta(x_1,x_2,x_3)^{\epsilon_{tm}(i)\epsilon_b(i)}=\theta(x_1,x_2,x_3)^{\epsilon(\tau_i)}=W^B_\theta(\tau_i,\mathcal C)$ and $W^B_\theta(j,\mathcal C)=\theta(x_1,x_2,x_3)^{\epsilon_{tm}(j)\epsilon_b(j)}=\theta(x_1,x_2,x_3)^{-\epsilon(\tau_j')}=W^B_\theta(\tau_j',\mathcal C)^{-1}$. This completes the proof of Claim, and therefore the proof of Theorem~\ref{thm-cocyclem}.
\end{proof}

\begin{example}\label{examp-1}
Let $\tau^3(3_1)$ be the $3$-twist-spun trefoil, which is presented by a marked graph diagram $D_3$ in Fig.~\ref{fig-3tw}. 

\begin{figure}[ht]
\begin{center}
\resizebox{0.5\textwidth}{!}{%
\includegraphics{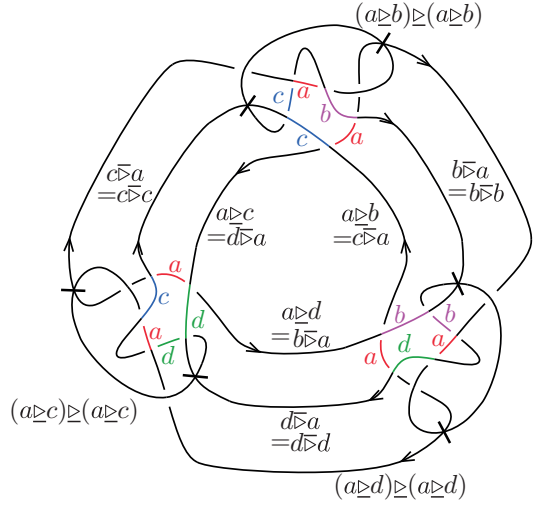}}
\caption{A marked graph diagram $D_3$ of the $3$-twist-spun trefoil}\label{fig-3tw}
\end{center}
\end{figure}

Let $X=\{1,2,3,4\}$ be a non-quandle biquandle in \cite{NeVo} with the matrix:
\begin{equation*}
M_X=\left[\begin{array}{cccc|cccc} 1& 4& 2& 3& 1& 1& 1& 1\\2& 3& 1& 4& 3& 3& 3& 3\\3& 2& 4& 1& 4& 4& 4& 4\\4& 1& 3& 2& 2& 2& 2& 2
\end{array} \right]. 
\end{equation*} 
It follows from Fig.~\ref{fig-3tw} that ${\rm Col}^B_X(D_3)$ is identified with the set $C$ of quadruples $(a,b,c,d)$ with $a,b,c,d \in X$, where $C=\{(1,1,1,1), (1,2,3,4), (1,3,4,2),$ $(1,4,2,3),$ $(2,1,4,3),$ $(2,2,2,2), (2,3,1,4), (2,4,3,1), (3,1,2,4), (3,2,4,1), (3,3,3,3),$ $(3,4,1,2),$ $(4,1,3,2), (4,2,1,3), (4,3,2,1), (4,4,4,4)\}$. 
This gives $\sharp{\rm Col}^B_X(\tau^3(3_1))=16$. Let $\theta=\chi_{(1,4,1)}\chi_{(1,4,3)}\chi_{(2,4,1)}\chi_{(2,4,3)}\chi_{(3,2,1)}
\chi_{(3,2,3)}\chi_{(4,2,1)}\chi_{(4,2,3)}$, where $\chi_{(a,b,c)}{(x,y,z)}$ is defined to be $t$ if ${(x,y,z)}={(a,b,c)}$ and 1 otherwise. Then it is seen that $\theta$ is a biquandle 3-cocycle with the coefficients in $\mathbb Z_2=<t \mid t^2=1>$.

\begin{figure}
\begin{center}
\resizebox{0.9\textwidth}{!}{%
\includegraphics{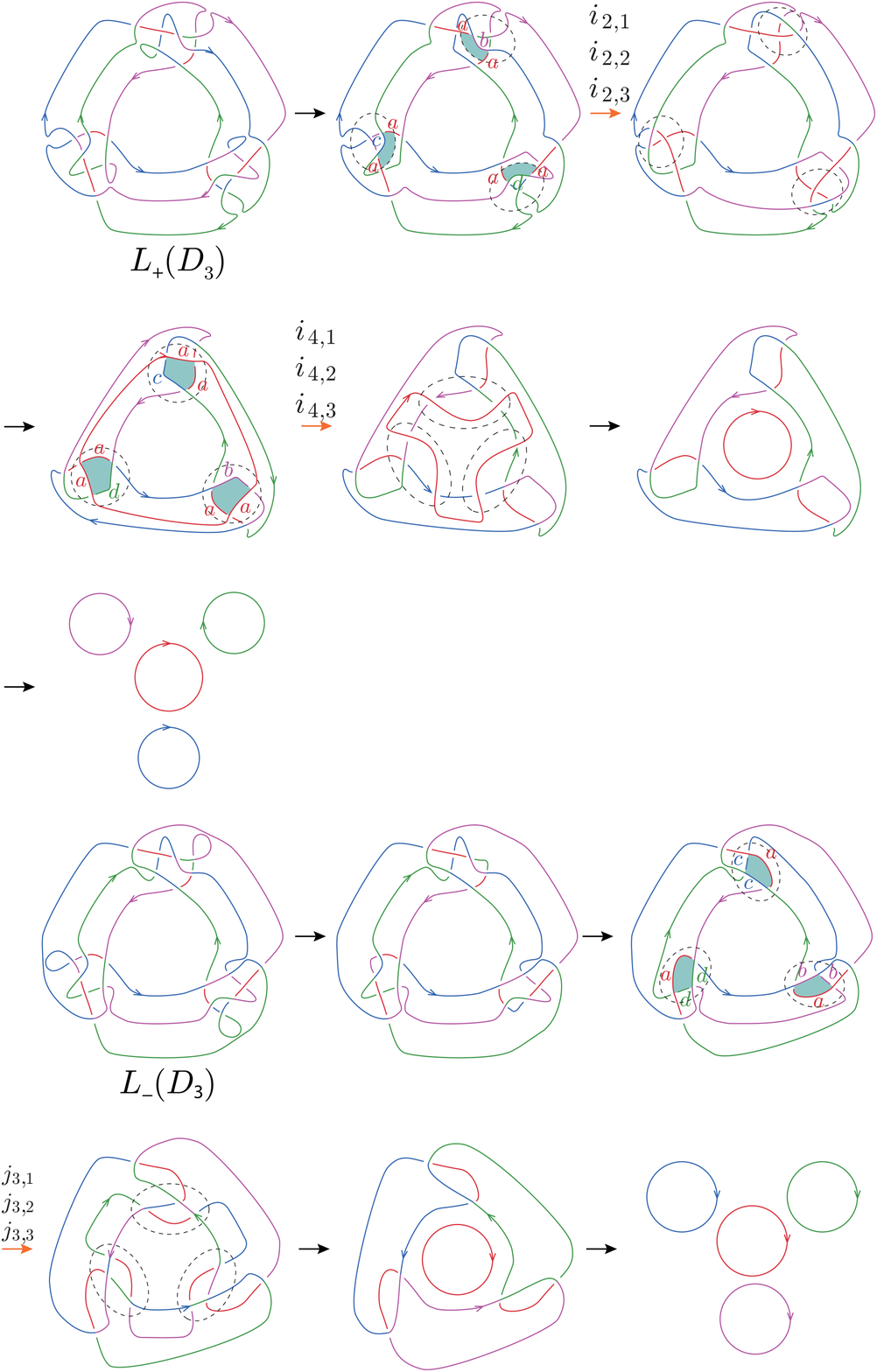}}
\caption{A sequence of link diagrams for two resolutions of $D_3$}\label{fig-bol}
\end{center}
\end{figure}

To compute $\Phi^B_\theta(\tau^3(3_1);\mathbb Z_2)$, we consider two sequences $L_+(D_3)=D_1\to D_2\to \cdots \to D_7=O^4$ and $L_-(D_3)=D_1'\to D_2'\to \cdots \to D_6'=O^4$ depicted in Fig.~\ref{fig-bol}. From the sequences, $\mathcal B(D_3)$ has 9 triple points. Let $i_{a,1}$, $i_{a,2}$ and $i_{a,3}$ be the triple points corresponding to the Reidemeister move $R_3$ between the upper parts, lower left parts, and lower right parts of $D_a$ and $D_{a+1}$ for $a=2,4$, respectively. Let $j_{3,1}$, $j_{3,2}$ and $j_{3,3}$ be the triple points corresponding to the Reidemeister move $R_3$ between the upper parts, lower left parts, and lower right parts of $D_3'$ and $D_4'$, respectively, as indicated in Fig.~\ref{fig-bol}. Then $I^3_+=\{i_{2,1},i_{2,2},i_{2,3},i_{4,1},i_{4,2},i_{4,3}\}, I^3_-=\{j_{3,1},j_{3,2},j_{3,3}\}.$ For the biquandle coloring $\mathcal C$ corresponding to $(a,b,c,d)$, the (Boltzman) weights are given by
\begin{equation*}
\begin{array}{lll} 
W^B_\theta(i_{2,1},\mathcal C)=\theta(a,a,b),& W^B_\theta(i_{2,2},\mathcal C)=\theta(a,a,c),& W^B_\theta(i_{2,3},\mathcal C)=\theta(a,a,d),\\
W^B_\theta(i_{4,1},\mathcal C)=\theta(a,c,a),& W^B_\theta(i_{4,2},\mathcal C)=\theta(a,d,a),& W^B_\theta(i_{4,3},\mathcal C)=\theta(a,b,a),\\
W^B_\theta(j_{3,1},\mathcal C)=\theta(c,a,c),& W^B_\theta(j_{3,2},\mathcal C)=\theta(d,a,d),& W^B_\theta(j_{3,3},\mathcal C)=\theta(b,a,b).
\end{array}
\end{equation*}
Therefore
\begin{align*} 
\Phi^B_\theta(\tau^3(3_1);\mathbb Z_2)&=\Phi^B_\theta(\mathcal B(D_3);\mathbb Z_2)=\sum_{\mathcal{C}\in {\rm Col}^B_X(D_3)}\biggl(\prod_{i\in I^3_+ }W^B_\theta(i,\mathcal{C})\prod_{j\in I^3_- }W^B_\theta(j,\mathcal{C})^{-1}\biggr)\\
&=\displaystyle\sum_{(a,b,c,d)\in C}\Big(\theta(a,a,b)\theta(a,a,c)\theta(a,a,d)\theta(a,c,a)\theta(a,d,a)\\
&\hspace{2cm}\theta(a,b,a)\theta(c,a,c)^{-1}\theta(d,a,d)^{-1}\theta(b,a,b)^{-1}\Big)\\
&=4+12t \in \mathbb Z[\mathbb Z_2].
\end{align*}
\end{example}

%%%

\section{Shadow biquandle cocycle invariants of surface-links}
\label{sect-scicl}

In \cite{CKS}, J. S. Carter, S. Kamada and M. Saito introduced the {\it shadow quandle cocycle invariants} for classical links and surface-links (including more general cases) by using the shadow cohomology theory of quandles, which are generalizations of quandle cocycle invariants. These invariants for links and surface-links are defined as state-sums over all quandle colorings of arcs and sheets together with particularly designed region colorings by use of Boltzman weights that are evaluations of a fixed cocycle at crossings of link diagrams and triple points of broken surface diagrams, respectively. In \cite{Le4}, S. Y. Lee introduced state-sum invariants for certain equivalence classes of cobordism surfaces in $\mathbb R^4$ between links by using the shadow cohomology theory of biquandles, which give shadow biquandle cocycle invariants for links as a special case.

In this section we construct shadow biquandle cocycle invariants for surface-links. We begin with reviewing the shadow biquandle (co)homology groups. Let $X$ be a biquandle. The {\it associated group} $G_X$ of $X$ is the group with a group presentation: 
$$\langle a \in X \mid   a \cdot (\ut{b}{a}) = b \cdot (\lt{a}{b}) ~\mbox{for}~   a, b \in X \rangle.$$  
An {\it $X$-set} is a nonempty set $Y$ with a right action of the associated group $G_X$.   We denote by $\lt{y}{g}$ the image of an element $y\in Y$ by the action of $g\in G_X$.  
Let $C_n^{\rm BR}(X)_Y$ be the free abelian group generated by $(n+1)$-tuples $(y,x_1,\ldots,x_n)$ for $n\geq 0$, where $y\in Y$ and $x_1,\ldots,x_n\in X$. For $n<0$, we assume that $C_n^{\rm BR}(X)_Y=0$. Define a homomorphism $\partial_n:C_n^{\rm BR}(X)_Y\rightarrow C_{n-1}^{\rm BR}(X)_Y$ by
\begin{align}\label{op-abqd}
\partial_n(y,x_1,&\ldots,x_n)=\sum_{i=1}^{n}(-1)^{i}[(y,x_1,\ldots,x_{i-1},x_{i+1},\ldots,x_n)\notag\\
&-(\lt{y}{x_i},\lt{x_1}{x_i},\ldots,\lt{x_{i-1}}{x_i},\ut{x_{i+1}}{x_i},\ldots,\ut{x_n}{x_i})]
\end{align}
for $n\geq2$ and $\partial_n=0$ for $n\leq1.$ Then we see that $C_\ast^{\rm BR}(X)_Y=\{C_n^{\rm BR}(X)_Y,\partial_n\}$ is a chain complex.

Let $C_n^{\rm BD}(X)_Y$ be the subset of $C_n^{\rm BR}(X)_Y$ generated by $(y,x_1,\ldots,x_n)$ with $x_i=x_{i+1}$ for some $i\in\{1,\ldots,n-1\}$ if $n\geq2$; otherwise, let $C_n^{\rm BD}(X)_Y=0.$ If $X$ is a biquandle, then $\partial_n(C_n^{\rm BD}(X)_Y)\subset C_{n-1}^{\rm BD}(X)_Y$ and $C_\ast^{\rm BD}(X)_Y=\{C_n^{\rm BD}(X)_Y,\partial_n\}$ is a subcomplex of $C_\ast^{\rm BR}(X)_Y.$ Put $C_n^{\rm BQ}(X)_Y=C_n^{\rm BR}(X)_Y/C_n^{\rm BD}(X)_Y$ and $C_\ast^{\rm BQ}(X)_Y=\{C_n^{\rm BQ}(X)_Y,\partial_n\}$.

For an abelian group $A$, we define the chain and cochain complexes

\begin{align*}
&C_\ast^{\rm W}(X;A)_Y=C_\ast^{\rm W}(X)_Y\otimes A,& &\partial=\partial\otimes \text{id},
\\
&C^\ast_{\rm W}(X;A)_Y=\text{Hom}(C_\ast^{\rm W}(X)_Y, A),& &\delta=\text{Hom}(\partial,\text{id})
\end{align*}
in the usual way, where ${\rm W=BR, BD, BQ}$.

\begin{definition}
The $n$th {\it shadow birack homology group} and the $n$th {\it shadow birack cohomology group} of a birack/biquandle $X$ with coefficient group $A$ are defined by
\begin{align*}
H_n^{\rm BR}(X;A)_Y=H_n(C_\ast^{\rm BR}(X;A)_Y), H_{\rm BR}^n(X;A)_Y=H^n(C^\ast_{\rm BR}(X;A)_Y).
\end{align*}
The $n$th {\it shadow degeneration homology group} and the $n$th {\it shadow degeneration  cohomology group} of a biquandle $X$ with coefficient group $A$ are defined by
\begin{align*}
H_n^{\rm BD}(X;A)_Y=H_n(C_\ast^{\rm BD}(X;A)_Y), H_{\rm BD}^n(X;A)_Y=H^n(C^\ast_{\rm BD}(X;A)_Y).
\end{align*}
The $n$th {\it shadow biquandle homology group} and the $n$th {\it shadow biquandle  cohomology group} of a biquandle $X$ with coefficient group $A$ are defined by
\begin{align*}
H_n^{\rm BQ}(X;A)_Y=H_n(C_\ast^{\rm BQ}(X;A)_Y), H_{\rm BQ}^n(X;A)=H^n(C^\ast_{\rm BQ}(X;A)_Y).
\end{align*}
\end{definition}

The shadow $n$-cycle and $n$-boundary groups (resp. the shadow $n$-cocycle and $n$-coboundary groups) are denoted by $Z_n^{\rm W}(X;A)_Y$ and $B^{\rm W}_n(X;A)_Y$ (resp. $Z_{\rm W}^n(X;A)_Y$ and $B^n_{\rm W}(X;A)_Y$), so that 
\begin{align*}
&H_n^{\rm W}(X;A)_Y=Z_n^{\rm W}(X;A)_Y/B_n^{\rm W}(X;A)_Y,\\
&H_{\rm W}^n(X;A)_Y=Z_{\rm W}^n(X;A)_Y/B_{\rm W}^n(X;A)_Y,
\end{align*}
where ${\rm W}$ is one of ${\rm  BR, BD, BQ}$. We will omit the coefficient group $A$ if $A=\mathbb Z$.

\begin{lemma}\label{prop-sbq3co}
Let $X$ be a finite biquandle and let $Y$ be a nonempty $X$-set.
A homomorphism $\theta: C^{\rm BR}_3(X)_Y \to A$ is a 3-cocycle of the shadow biquandle cochain complex $C^\ast_{\rm BQ}(X;A)_Y$ if and only if $\theta$ satisfies the following two conditions:
\begin{itemize}
\item [(i)] $\theta(y,a,a,b)=0$ and $\theta(y,a,b,b)=0$ for all $y \in Y$ and $a, b \in X$.
\item [(ii)] $\theta(y,b,c,d)+\theta(y,a,b,d)+\theta(\lt{y}{b},\lt{a}{b},\ut{c}{b},\ut{d}{b})+\theta(\lt{y}{d},\lt{a}{d},\lt{b}{d},\lt{c}{d})$
$=\theta(y,a,c,d)+\theta(y,a,b,c)+\theta(\lt{y}{a},\ut{b}{a},\ut{c}{a},\ut{d}{a})+\theta(\lt{y}{c},\lt{a}{c},\lt{b}{c},\ut{d}{c})$ for all $y \in Y$ and $a, b, c, d \in X$.
\end{itemize}
\end{lemma}

\begin{proof}
Suppose that $\theta \in Z^3_{\rm BQ}(X;A)_Y$. Then $\theta(C_3^{\rm BD}(X)_Y)=0$ and $\delta(\theta)=\theta\circ\partial_4=0$. Since $(y,a,a,b), (y,a,b,b) \in C_3^{\rm BD}(X)_Y$ for all $y\in Y$ and $a,b\in X$, we obtain the condition (i).
For every $y \in Y$ and $a, b, c, d \in X$, $(y,a,b,c,d)\in C_4^{\rm BR}(X)_Y$ and hence $(\theta\circ\partial_4)((y,a,b,c,d))=0$. From (\ref{op-abqd}), we have \begin{equation}\label{eq-sbqcc1}
\begin{split}
\partial_4(y,a,b,c,d)&=(y,b,c,d)-(y,a,c,d)+(y,a,b,d)-(y,a,b,c)\\
&-(\lt{y}{a},\ut{b}{a},\ut{c}{a},\ut{d}{a})+(\lt{y}{b},\lt{a}{b},\ut{c}{b},\ut{d}{b})\\
&-(\lt{y}{c},\lt{a}{c},\lt{b}{c},\ut{d}{c})+(\lt{y}{d},\lt{a}{d},\lt{b}{d},\lt{c}{d}).
\end{split}
\end{equation}
By takng $\theta$ on both sides of (\ref{eq-sbqcc1}), we have the condition (ii).

Conversely, suppose that a homomorphism $\theta: C^{\rm BR}_3(X)_Y \to A$ satisfies the two conditions (i) and (ii). Since $C_3^{\rm BD}(X)_Y$ is generated by the elements $(y,a,a,b)$ and $(y,a,b,b)$ for $y\in Y$ and $a,b\in X$, it is direct from the condition (i) that $\theta(C_3^{\rm BD}(X)_Y)=0$. Now for every $(y,a,b,c,d)\in Y \times X^4,$ it is easily seen from (\ref{eq-sbqcc1}) and the condition (ii) that $(\theta\circ\partial_4)((y,a,b,c,d))=0$. Since $C^{\rm BR}_4(X)_Y$ is generated by the elements $(y,a,b,c,d)\in Y \times X^4,$ we see that $\delta(\theta)(C^{\rm BR}_4(X)_Y)=(\theta\circ\partial_4)(C^{\rm BR}_4(X)_Y)=0$, i.e., $\delta(\theta)=0$. This implies that $\theta \in Z^3_{\rm BQ}(X;A)_Y$. 
\end{proof}

The two conditions (i) and (ii) in Lemma~\ref{prop-sbq3co} are called the {\it shadow biquandle 3-cocycle condition}. A homomorphism $\theta : C^{\rm BR}_3(X)_Y \to A$ or a map $\theta : Y\times X^3 \to A$ satisfying the shadow biquandle 3-cocycle condition is called a {\it shadow biquandle $3$-cocycle}.

\bigskip

 Let $X$ be a finite biquandle and let $Y$ be a nonempty $X$-set. Let $\mathcal B$ be a broken surface diagram and let $R(\mathcal B)$ be the set of the complementary regions of $\mathcal B$ in $\mathbb R^3$. For a biquandle $X$-coloring $\mathcal{C}: S(\mathcal B) \rightarrow X$ of $\mathcal B$, a {\it shadow biquandle coloring} of $\mathcal B$ (extending a given biquandle coloring $\mathcal C$) by $(X,Y)$ or simply a {\it shadow biquandle $(X,Y)$-coloring} of $\mathcal B$ is a map $\tilde{\mathcal{C}}: S(\mathcal B) \cup  R(\mathcal B) \rightarrow X \cup Y $ satisfying the conditions:  
\begin{itemize}   
\item $\tilde{\mathcal{C}} (S(\mathcal B)) \subset X$ and $\tilde{\mathcal{C}} (R(\mathcal B)) \subset Y$. 
\item The restriction of $\tilde{\mathcal C}$ to $S(\mathcal B)$ is a given biquandle $X$-coloring $\mathcal C$.   
\item   If two adjacent regions $f_1$ and $f_2$ are separated by a semi-sheet $e$ 
and the co-orientation of $e$ points from $f_1$ to $f_2$, then $\tilde{\mathcal{C}}(f_2)=\lt{\tilde{\mathcal{C}}(f_1)} {\tilde{\mathcal{C}}(e)}$ (see Fig.~\ref{fig-sdpc}).  
\end{itemize}
We denote by ${\rm Col}^{\rm SB}_{(X,Y)}(\mathcal B)$ the set of all shadow biquandle $(X,Y)$-colorings of $\mathcal B$.

\begin{figure}[ht]
\begin{center}
\resizebox{0.9\textwidth}{!}{%
  \includegraphics{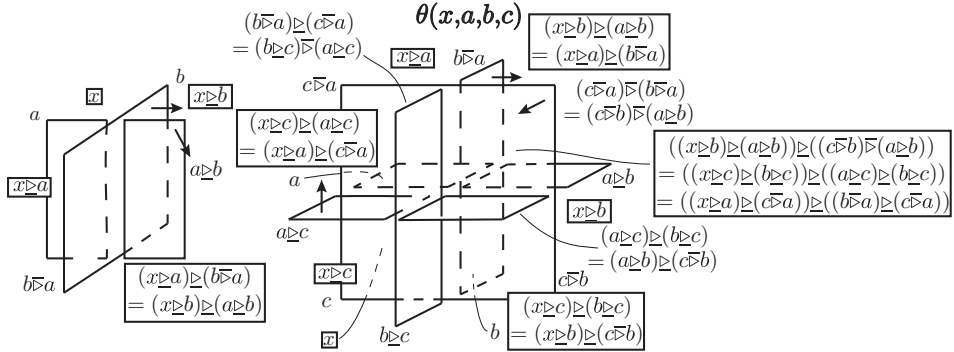}}
\caption{Shadow colors at a double point curve and a triple point}\label{fig-sdpc}
\end{center}
\end{figure}

\begin{theorem}\label{thm-sxycnbr}
Let $\mathcal L$ be a surface-link and let $\mathcal B$ and $\mathcal B'$ be two broken surface diagrams of $\mathcal L$. Then for any finite biquandle $X$ and a nonempty $X$-set $Y$, there is a one-to-one correspondence between ${\rm Col}^{\rm SB}_{(X,Y)}(\mathcal B)$ and ${\rm Col}^{\rm SB}_{(X,Y)}(\mathcal B')$. Consequently,  the cardinal number $\sharp{\rm Col}^{\rm SB}_{(X,Y)}(\mathcal B)$ is an invariant of $\mathcal L$.
\end{theorem}

\begin{proof} 
By Theorem \ref{thm-RosM}, it suffices to prove the assertion for the case that $\mathcal B'$ is obtained from $\mathcal B$ by a single Roseman move. Let $E$ be an open $3$-disk in $\mathbb R^3$ where a Roseman move under consideration is applied. Then $\mathcal B\cap (\mathbb R^3 - E) = \mathcal B'\cap (\mathbb R^3 -E)$. Now let $\mathcal C$ be a shadow biquandle $(X,Y)$-coloring of $\mathcal B$. Using biquandle axioms of Definition \ref{defn-biqdle}, the definition of an $X$-set and Fig.~\ref{fig-sdpc},  it is seen that for each Roseman move, the restriction of $\mathcal C$ to $\mathcal B\cap (\mathbb R^3 - E) (=\mathcal B'\cap (\mathbb R^3 -E))$ can be uniquely extended to a biquandle $(X,Y)$-coloring $\mathcal C'$ of $\mathcal B'$, and conversely the restriction of the unique biquandle $(X,Y)$-coloring $\mathcal C'$ to $\mathcal B'\cap (\mathbb R^3 -E)$ is extended to the biquandle $(X,Y)$-coloring $\mathcal C$. 
\end{proof}

We call the cardinal number $\sharp{\rm Col}^{\rm SB}_{(X,Y)}(\mathcal B)$ the {\it shadow biquandle $(X,Y)$-coloring number} of $\mathcal L$ and denote it by $\sharp{\rm Col}^{\rm SB}_{(X,Y)}(\mathcal L)$.

\bigskip

Let $\theta$ be a shadow biquandle $3$-cocycle and let $\tilde{\mathcal C}$ be a shadow biquandle $(X,Y)$-coloring of a broken surface diagram $\mathcal B$. Let $R$ be the source region of a triple point $\tau$ of $\mathcal B$. If $x$ is the color of the source region $R$ and $a, b$ and $c$ are the colors of the bottom, middle and top semi-arcs facing $R$, respectively, as depicted in Fig.~\ref{fig-sdpc}. Then we define the {\it shadow Boltzman weight} $W^{\rm SB}_\theta(\tau,\tilde{\mathcal C})$ at $\tau$ with respect to $\tilde{\mathcal C}$ by  $$W^{\rm SB}_\theta(\tau,\tilde{\mathcal C})=\theta(x,a,b,c)^{\epsilon(\tau)}.$$

\begin{figure}[ht]
\begin{center}
\resizebox{0.60\textwidth}{!}{%
  \includegraphics{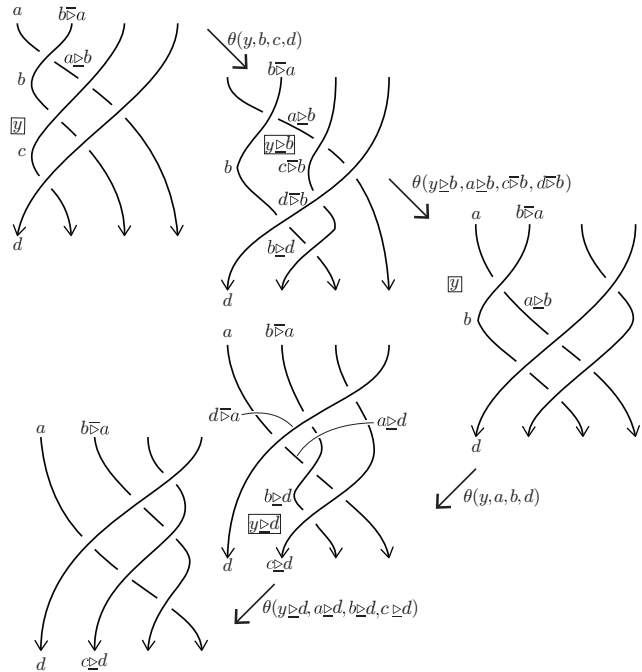}}
\caption{The tetrahedral move and shadow biquandle 3-cocycle condition (ii), LHS}\label{fig-shlcol1}
\end{center}
\end{figure}

\begin{definition}\label{def-shlci}
Let $\theta$ be a shadow biquandle $3$-cocycle. The {\it shadow state-sum} or {\it shadow partition function} (associated with $\theta$) of a broken surface diagram $\mathcal B$ is defined by $$\Phi^{\rm SB}_\theta(\mathcal B; A)=\sum_{\tilde{\mathcal C}\in {\rm Col}^{\rm SB}_{X}(\mathcal B)_Y}\prod_{\tau \in T(\mathcal B)}W^{\rm SB}_\theta(\tau,\tilde{\mathcal C}) \in \mathbb Z[A].$$ 
\end{definition}

\begin{theorem}\label{thm-shinv}
Let $\mathcal L$ be a surface-link and let $\mathcal B$ be a broken surface diagram of $\mathcal L$. For any given shadow biquandle $3$-cocycle $\theta$, the shadow state-sum $\Phi^{\rm SB}_\theta(\mathcal B; A)$ is an invariant of $\mathcal L$.  (It is denoted by $\Phi^{\rm SB}_\theta(\mathcal L; A)$.)   
\end{theorem}

\begin{proof}
Suppose that $\mathcal B'$ is a broken surface diagram obtained from $\mathcal B$ by a single Roseman move. For each shadow biquandle $(X,Y)$-coloring $\tilde{\mathcal C}$ of $\mathcal B$, let $\tilde{\mathcal C'}$ be the corresponding shadow biquandle $(X,Y)$-coloring of $\mathcal B'$ as in the proof of Theorem \ref{thm-sxycnbr}. It suffices to verify that $\prod_{\tau \in T(\mathcal B)}W^{\rm SB}_\theta(\tau,\tilde{\mathcal C})=\prod_{\tau \in T(\mathcal B')}W^{\rm SB}_\theta(\tau,\tilde{\mathcal C'})$. It is direct from Fig.~\ref{fig-rose} that the Roseman moves of type 1, 2, 3 and 4 involve no triple points. For the Roseman move of type 5, the product of weights differ by $\phi(y,x_1,x_1,x_2)^{\pm1}$ or $\phi(y,x_1,x_2,x_2)^{\pm1}$ for some $y\in Y$ and $x_1,x_2 \in X$ and it is immediate from Lemma~\ref{prop-sbq3co} (i) that $\phi(y,x_1,x_1,x_2)^{\pm1}=\phi(y,x_1,x_2,x_2)^{\pm1}=1$. Hence the product of the weights are not changed. For the Roseman move of type 6, two triple points with the same weights of opposite signs are involved and the product of the weights are cancelled. For the Roseman move of type 7, there are four involved triple points before and after the move as illustrated in Figs.~\ref{fig-shlcol1} and \ref{fig-shlcol2} in motion pictures (the tetrahedral move). From Lemma~\ref{prop-sbq3co} (ii), we see that the product of the weights are unchanged. 
\end{proof}

\begin{figure}[ht]
\begin{center}
\resizebox{0.60\textwidth}{!}{%
  \includegraphics{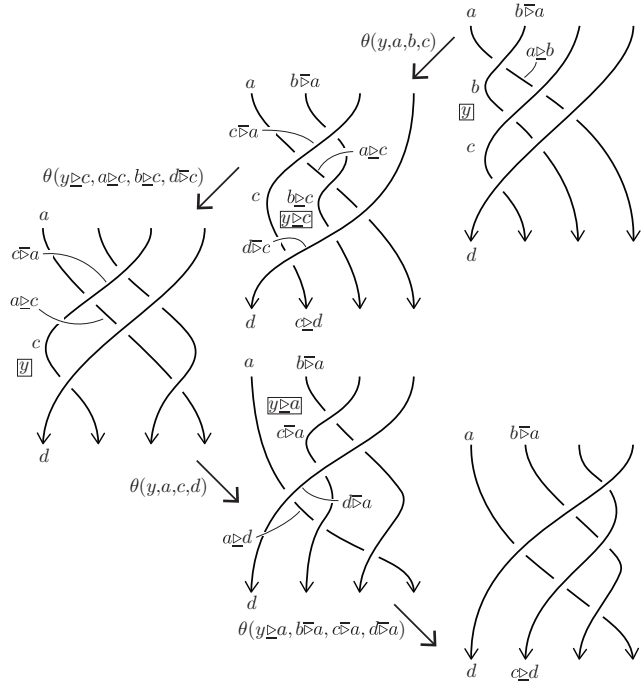}}
\caption{The tetrahedral move and shadow biquandle 3-cocycle condition (ii), RHS}\label{fig-shlcol2}
\end{center}
\end{figure}

By a similar argument as in the proof of \cite[Proposition 5.7]{CJKLS}, it is verified that if $\Phi^{\rm SB}_\theta$ and $\Phi^{\rm SB}_{\theta'}$ are the shadow state-sum invariants defined from cohomologous shadow biquandle cocycles $\theta$ and $\theta'$ (so that $\theta=\theta'\delta\phi$ for some shadow $2$-cochain $\phi$), then $\Phi^{\rm SB}_\theta(\mathcal L)=\Phi^{\rm SB}_{\theta'}(\mathcal L)$ for any surface-link $\mathcal L$. In particular, if $\theta$ is a shadow coboundary, then $\Phi^{\rm SB}_{\theta}(\mathcal L)$ is trivial for all surface-link $\mathcal L$.

%%%

\section{Shadow biquandle cocycle invariants from marked\\ graph diagrams}\label{sect-sqcocm} 

In this section we introduce a method of computing shadow biquandle $3$-cocycle invariants of surface-links from marked graph diagrams.  

\begin{definition}\label{defn-sqc-1} 
Let $X$ be a finite biquandle and let $Y$ be a nonempty $X$-set. Let $D$ be a marked graph diagram in $\mathbb R^2$ with co-orientation, $R(D)$ the set of the complementary regions of $D$ in $\mathbb R^2$, and let $\mathcal C$ be a biquandle $X$-coloring of $D$. A {\it shadow biquandle coloring} of $D$ by $(X,Y)$ or a {\it shadow biquandle $(X,Y)$-coloring} of $D$ (extending $\mathcal C$) is a map $\tilde{\mathcal{C}}: S(D)\cup R(D) \to X\cup Y$ satisfying the following conditions (1), (2), and (3): 
\begin{itemize} 
\item [(1)] $\tilde{\mathcal{C}}(S(D) )\subset X$ and $\tilde{\mathcal{C}}(R(D) )\subset Y.$ 
\item [(2)] The restriction of $\tilde{\mathcal{C}}$ to $S(D)$ is a given biquandle $X$-coloring $\mathcal C.$ 
\item[(3)] If two adjacent regions $f_1$ and $f_2$ are separated by a semi-arc $e$  
and the co-orientation of $e$ points from $f_1$ to $f_2$, then $\tilde{\mathcal{C}}(f_2)=\lt{ \tilde{\mathcal{C}}(f_1)} {\tilde{\mathcal{C}}(e)}$ (see Fig.~\ref{fig-col}).     
\begin{figure}[ht]
\begin{center}
\resizebox{0.75\textwidth}{!}{%
  \includegraphics{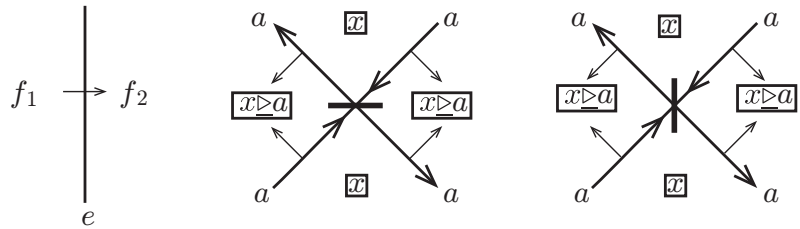}}
\caption{Colorings of regions}\label{fig-col} 
\end{center}
\end{figure}
\end{itemize} 
\end{definition} 

We denote by ${\rm Col}^{\rm SB}_{(X,Y)}(D)$ the set of all shadow biquandle $(X,Y)$-colorings of $D$.

\begin{theorem}\label{thm-shbq-col}
Let $\mathcal L$ be a surface-link and let $D$ and $\mathcal B$ be a marked graph diagram and a broken surface diagram presenting $\mathcal L$, respectively. Then there is a one-to-one correspondence between ${\rm Col}^{\rm SB}_{(X,Y)}(D)$ and ${\rm Col}^{\rm SB}_{(X,Y)}(\mathcal B)$. Consequently, $\sharp{\rm Col}^{\rm SB}_{(X,Y)}(\mathcal B)=\sharp{\rm Col}^{\rm SB}_{(X,Y)}(D)$.
\end{theorem}

\begin{proof}
By Theorem \ref{thm-sxycnbr}, we may assume that $\mathcal B$ is a broken surface diagram $\mathcal B(D)$ associated with $D$ defined in Section \ref{sect-mgd}. Let $\tilde{\mathcal{C}}$ be a shadow biquandle $(X,Y)$-coloring of $\mathcal B$. Then the $0$-level cross-section with colors inherited from $\tilde{\mathcal{C}}$ is a shadow biquandle $(X,Y)$-coloring of $D$. This gives a map from ${\rm Col}^{\rm SB}_{(X,Y)}(\mathcal B)$ to ${\rm Col}^{\rm SB}_{(X,Y)}(D)$. Conversely, using the same argument as in \cite{As}, we obtain the inverse map from ${\rm Col}^{\rm SB}_{(X,Y)}(D)$ to ${\rm Col}^{\rm SB}_{(X,Y)}(\mathcal B)$. 
\end{proof}

 Let $D$ be a marked graph diagram and let $D_+$ and $D_-$ be the positive and negative resolution of $D$, respectively. Let $D_+=D_1 \rightarrow D_2 \rightarrow\cdots\rightarrow D_r=O$ and $D_-=D_1' \rightarrow D_2' \rightarrow\cdots\rightarrow D_s'=O'$ be sequences of link diagrams from $D_+$ and $D_-$ to trivial link diagrams $O$ and $O'$, respectively, related by ambient isotopies of $\mathbb R^2$ and Reidemeister moves. Let $I^3_+, {I^3_-}, \epsilon_{tm}(i)$ and $\epsilon_{b}(i)$ be the same as in Section~\ref{sect-qcocm}.
Let $\tilde{\mathcal{C}}: S(D)\cup R(D) \to X\cup Y$ be a shadow biquandle $(X,Y)$-coloring of $D$ and let $\theta$ be a shadow biquandle $3$-cocycle. 
Let $i\in I^3_+ \amalg {I^3_-}$ and let $R$ be the source region of the stage $i$. The {\it shadow (Boltzman) weight} $W^{\rm SB}_\theta(i,\tilde{\mathcal{C}})$ at $i$ with respect to $\tilde{\mathcal C}$ is defined by $$W^{\rm SB}_\theta(i,\tilde{\mathcal{C}})=\theta(y,x_1,x_2,x_3)^{\epsilon_{tm}(i)\epsilon_{b}(i) },$$ where $y$ is the color of the source region $R$ and $x_1, x_2$ and $x_3$ are the colors of the bottom, middle and top semi-arcs facing $R$ at the stage $i$, respectively, as illustrated in Fig.~\ref{fig-slab1}. 
\begin{figure}[ht] 
\begin{center}
\resizebox{0.80\textwidth}{!}{% 
\includegraphics{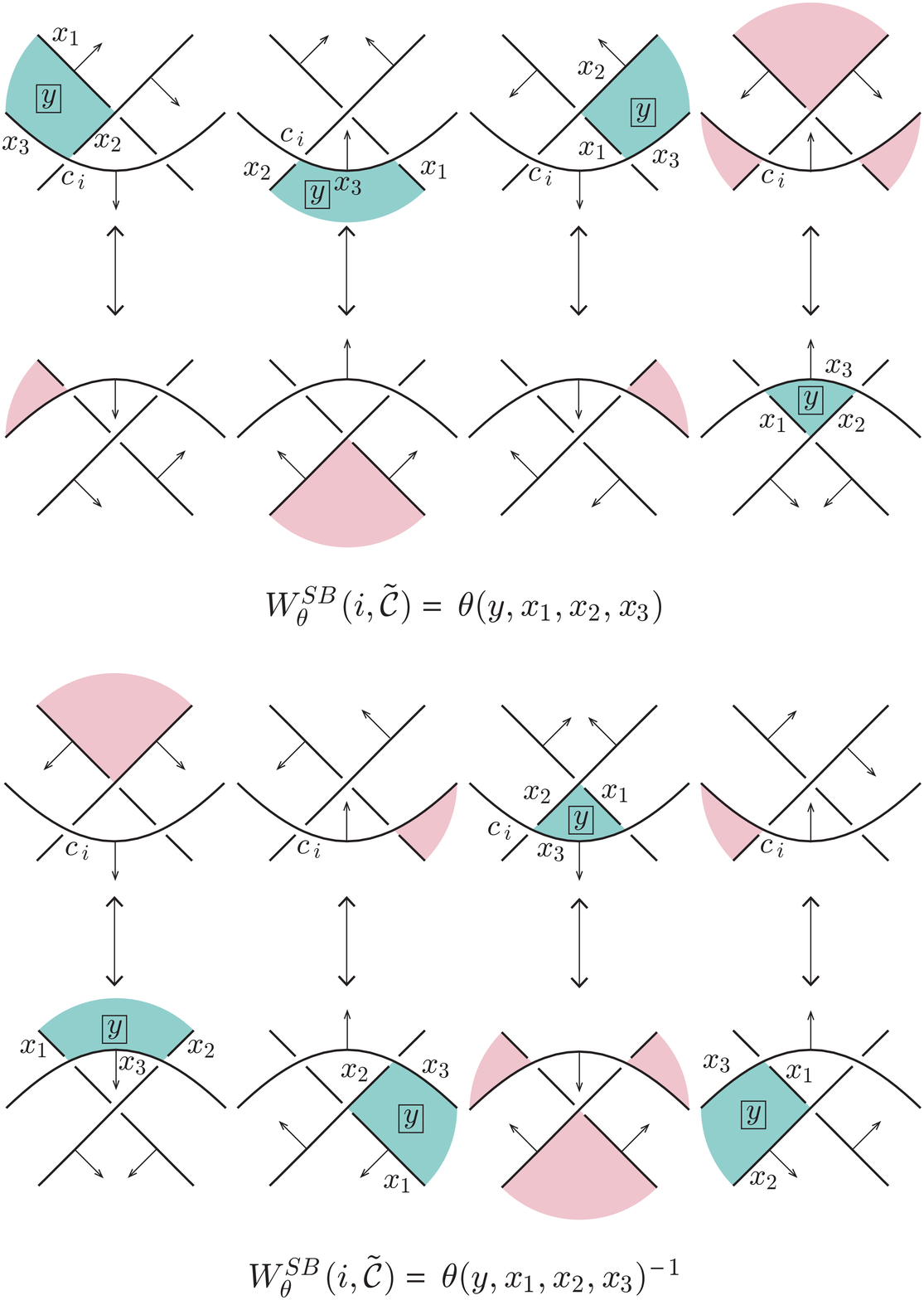} } 
\caption{Shadow Boltzman weight  at $i\in I^3_+ \amalg {I^3_-}$}\label{fig-slab1} 
\end{center} 
\end{figure}  

\begin{definition}\label{def-cocyclesm} 
Let $D$ be a marked graph diagram. The {\it shadow state-sum} or {\it shadow partition function} of $D$ (associated with  $\theta$) is defined by 
\begin{align*} 
\Phi_\theta^{\rm SB}(D; A)=\sum_{\tilde{\mathcal C}\in {\rm Col}^{\rm SB}_{(X,Y)}(D)}\biggl(\prod_{ i\in I^3_+ }W^{\rm SB}_\theta(i,\tilde{\mathcal C})\prod_{j\in I^3_- }W^{\rm SB}_\theta(j,\tilde{\mathcal C})^{- 1}\biggr).\end{align*} 
\end{definition}  

\begin{theorem}\label{thm-cocyclesm} 
Let $\mathcal L$ be a surface-link and let $D$ be a marked graph diagram of $\mathcal L$. Then for any shadow biquandle $3$-cocycle $\theta,$ $\Phi_\theta^{\rm SB}(\mathcal L; A)=\Phi_\theta^{\rm SB}(D; A)$.  
\end{theorem} 

\begin{proof}
By the same argument as the proof of Theorem~\ref{thm-cocyclem}, the assertion follows.
\end{proof}

%%%%

\section*{Acknowlegements}
The first and second authors were supported by JSPS KAKENHI Grant Numbers JP26287013 and JP24244005. The third author was supported by JSPS overseas post doctoral fellow Grant Number JP15F15319. The fourth author was supported by Basic Science Research Program through the National Research Foundation of Korea (NRF) funded by the Ministry of Education, Science and Technology (NRF-2016R1A2B4016029). 

%%%

\end{document}